\documentclass[aos,preprint]{imsart}

\RequirePackage[OT1]{fontenc}
\RequirePackage{amsthm,amsmath,graphicx}
\RequirePackage[numbers]{natbib}
\RequirePackage[colorlinks,citecolor=red,urlcolor=blue]{hyperref}

\usepackage{amsfonts,amssymb}
\usepackage[]{algorithm2e}
\usepackage[usenames,dvipsnames]{xcolor}

\arxiv{arXiv:0000.0000}

\startlocaldefs

\numberwithin{equation}{section}
\theoremstyle{plain}
\newtheorem{lemma}{Lemma}[section]
\newtheorem{proposition}{Proposition}[section]
\newtheorem{theorem}{Theorem}[section]

\newtheorem{corollary}{Corollary}[section]

\newcommand{\p} {\mathbb{P}}
\newcommand{\E} {\mathbb{E}}
\newcommand{\R} {\E r}

\DeclareMathOperator*{\argmin}{arg\,min}
\DeclareMathOperator*{\argmax}{arg\,max}

\let\oldenumerate\enumerate
\renewcommand{\enumerate}{
  \oldenumerate
  \setlength{\itemsep}{1pt}
  \setlength{\parskip}{0pt}
  \setlength{\parsep}{0pt}
}

\newcommand{\Bin} {\text{Bin}} % Binomial
\newcommand{\Ber} {\text{Ber}} % Bernoulli

\newcommand{\nb}[1]{}

\endlocaldefs

\begin{document}
\begin{frontmatter}
\title{Minimax Rates of Community Detection in \\Stochastic Block Models}
\runtitle{Minimax Rates of Community Detection in SBM}

\begin{aug}
\author{\fnms{Anderson Y.} \snm{Zhang}\ead[label=e1]{ye.zhang@yale.edu}},
\and
\author{\fnms{Harrison H.} \snm{Zhou}\ead[label=e2]{huibin.zhou@yale.edu}\ead[label=u1,url]{http://www.stat.yale.edu/\textasciitilde hz68/}}
\runauthor{Zhang, Zhou}
\affiliation{Yale University}
\address{Department of Statistics\\
Yale University\\
New Haven, CT 06511 \\
\printead{e1}\\
\printead{e2}\\
\printead{u1}}
\end{aug}

\begin{abstract}
Recently network analysis has gained more and more attentions in statistics, as well as in computer science, probability, and applied mathematics. Community detection for the stochastic block model (SBM) is probably the most studied topic in network analysis. Many methodologies have been proposed. Some beautiful and significant phase transition results are obtained in various settings. In this paper, we provide a general minimax theory for community detection. It gives minimax rates of the mis-match ratio for a wide rage of settings including homogeneous and inhomogeneous SBMs, dense and sparse networks, finite and growing number of communities. The minimax rates are exponential, different from polynomial rates we often see in statistical literature. An immediate consequence of the result is to establish threshold phenomenon for strong consistency (exact recovery) as well as weak consistency (partial recovery). We obtain the upper bound by a range of penalized likelihood-type approaches. The lower bound is achieved by a novel reduction from a global mis-match ratio to a local clustering problem for one node through an exchangeability property.
\end{abstract}



\end{frontmatter}

\nb{The table of contents is temporarily included.}
%\tableofcontents

\section{Introduction}
Network science \citep{easleynetworks, newman2010networks, van2006performance, lovasz2012large} has become one of the most active research areas over the past few years. It has applications in many disciplines, for example, physics \citep{newman2003structure}, sociology \citep{wasserman1994social}, biology \citep{barabasi2004network}, and Internet \citep{albert1999internet}. Detecting and identifying communities is fundamentally important to understand the underlying structure of the network \citep{girvan2002community}. Many models and methodologies have been proposed for community detection from different perspectives, including RatioCut\citep{hagen1992new}, Ncut \citep{shi2000normalized}, and spectral method \citep{mcsherry2001spectral, rohe2011spectral, lei2013consistency} from computer science, Newman\textendash Girvan Modularity \citep{girvan2002community} from physics, semi-definite programming \citep{cai2014robust, hajek2015achieving} from engineering, and maximum likelihood estimation \citep{amini2013pseudo, bickel2009nonparametric} from statistics.

Deep theoretical developments have been actively pursued as well. Recently, celebrated works of Mossel et al. \citep{mossel2012stochastic, mossel2013proof} and Massoulie \citep{massoulie2014community} considered balanced two-community sparse networks, and discovered the threshold phenomenon for both weak and strong consistency of community detection. Further extensions to slowly growing number of communities have been made in \citep{hajek2015achieving, mossel2014consistency, chen2014statistical, abbe2015community}. Recently in statistical literature, theoretical properties of various methods had been investigated as well in \citep{chen2014statistical, zhao2012consistency, bickel2013asymptotic, chin2015stochastic, rohe2011spectral, lei2013consistency}, usually under weaker conditions and better suited for real data applications, but the convergence rates may often be sub-optimal.

Despite recent active and significant developments in network analysis, assumptions and conclusions can be very different in different papers. There is not an integrated framework on optimal community detection. In this paper, we attempt to give a fundamental and unified understanding of the community detection problem for the Stochastic Block Model (SBM). Our framework is quite general, including homogeneous and inhomogeneous SBMs, dense and sparse networks, equal and non-equal community sizes, and finite and growing number of communities. For example, the connection probability can be as small as an order of $1/n$, or as large  as a constant order, and the total number of communities can be as large as $n/\log n$. Under this framework, a sharp minimax result is obtained with an exponential rate. This result gives a clear and smooth transition from weak consistency (partial recovery) to strong consistency (exact recovery), i.e., clustering error rates from $o(1)$ to $o(n^{-1})$. As a consequence, we obtain phase transitions for non-consistency and strong consistency, under various settings, which recover the tight thresholds for phase transition in \citep{mossel2012stochastic, mossel2013proof, mossel2014consistency, chen2014statistical}.

The Stochastic Block Model, proposed by \citep{holland1983stochastic}, is possibly the most studied model in community detection \citep{bickel2009nonparametric, rohe2011spectral, lei2013consistency}. Consider an undirected network with totally $n$ nodes, and $K$ communities labeled as $\{1,2\ldots,K\}$. Each node is assigned to one community. Denote $\sigma$ to be an assignment, and $\sigma(i)$ is the community assignment for the $i$-th node. Let $n_k=|\{i:\sigma(i)=k\}|$ be the size of the $k$-th community, for each $k\in\{1,2,\ldots, K\}$. We observe the connectivity of the network, which is encoded into the adjacency matrix $\{A_{i,j}\}$ taking values in $\{0,1\}^{n\times n}$. If there exists a connection between two nodes, $A_{i,j}$ is equal to 1, and 0 otherwise. We assume each $A_{i,j}$ for any $i\geq j$ to be an independent Bernoulli random variable with a success probability $\theta_{i,j}$. Let $A_{i,i}=0$ (no self-loop) and $A_{i,j}=A_{j,i}$ (symmetry) for any $i,j$. In the SBM, $\{\theta_{i,j}\}$ is assumed to have a blockwise structure, in the sense that $\theta_{i,j}=\theta_{i',j'}$ if $i$ and $i'$ are from the same community, and so are $j$ and $j'$. We require that the within-community probabilities are larger than the between-communities probabilities, as in reality individuals from the same community are often more likely to be connected.

We consider a general SBM with parameter space defined as follows,
\begin{align*}
\Theta(n&,K,a,b,\beta)\triangleq\Bigg\{(\sigma,\{\theta_{i,j}\}):\;\sigma:[n]\rightarrow [K]^n,n_k\in\Big[\frac{n}{\beta K},\frac{\beta n}{K}\Big],\forall k\in[K],\;\{\theta_{i,j}\}\in[0,1]^{n\times n},\\
&\theta_{i,j}\geq  \frac{a}{n} \text{ if }\sigma(i)=\sigma(j)\text{ and }\theta_{i,j}\leq  \frac{b}{n} \text{ if }\sigma(i)\neq\sigma(j),\;\theta_{i,i}=0,\;\theta_{i,j}=\theta_{j,i},\;\forall i\neq j\Bigg\},
\end{align*}
where $\beta\geq 1$ and is bounded. When $\beta=1+o(1)$, all communities have almost the same size. The parameters $a/n$ and $b/n$ have straightforward interpretation, with the former one as the smallest within-community probability and the later as the largest between-community probability. Throughout the paper, we assume $\epsilon<b<a$ and $a/n<1-\epsilon$ for a small constant $\epsilon>0$, allowing the network to be from very sparse to very dense.

We use the mis-match ratio $r(\sigma,\hat\sigma)$ to measure the performance of community detection. It is the proportion of nodes mis-clustered by $\hat\sigma$ against the truth $\sigma$. The exact definition is given in Section \ref{subsec::mis_match}. The minimax rate for the parameter space $\Theta(n,K,a,b,\beta)$ in terms of the mis-match ratio loss is as follows.
\begin{theorem}\label{thm:main}
Assume $\frac{nI}{K\log K}\rightarrow\infty$, then
\begin{align}\label{eqn:main}
\inf_{\hat\sigma}\sup_{\Theta(n,K,a,b,\beta)}\mathbb{E}r(\sigma,\hat\sigma)=
\begin{cases}
\exp\big(-(1+o(1))\frac{nI}{2}\big),K=2,\\
\exp\big(-(1+o(1))\frac{nI}{\beta K}\big),K\geq 3,
\end{cases}
 \end{align}
where $1\leq\beta<\sqrt{5/3}$. In addition, if $nI/K=O(1)$, there are at least a constant proportion of nodes mis-clustered, i.e., $\inf_{\hat\sigma}\sup_{\Theta(n,K,a,b,\beta)}\E r(\sigma,\hat\sigma)\geq c$, for some fixed constant $c>0$.
\end{theorem}
\noindent Note that when $K$ is finite, $nI\rightarrow \infty$ is a sufficient condition to get Equation (\ref{eqn:main}) since it is equivalent to $\frac{nI}{K\log K}\rightarrow\infty$.
Here the key quantity $I$ is defined as
\begin{align}
I=-2\log \Bigg(\sqrt{ \frac{a}{n}\frac{b}{n} }+\sqrt{1- \frac{a}{n} }\sqrt{1- \frac{b}{n} }\Bigg),
\end{align}
which is exactly $D_{1/2}(\Ber(\frac{a}{n})\|\Ber(\frac{b}{n}))$, the R\'{e}nyi divergence of order $1/2$ between two Bernoulli distributions $\Ber(\frac{a}{n})$ and $\Ber(\frac{b}{n})$. The form of $I$ is closely related to the Hellinger distance between those two Bernoulli probability measures. It is worth pointing out that $I$ is equal to $(a-b)^2/(an)$, up to a constant factor, which can be interpreted as the signal-to-noise ratio, as long as $a/n\leq 1-\epsilon$ for some $\epsilon>0$. In particular, when $a=o(n)$, $I$ is equal to $(1+o(1))(\sqrt{a}-\sqrt{b})^2/n$. 

The lower bound of (\ref{eqn:main}) is achieved by a novel reduction of the \textit{global} minimax rate into a \textit{local} testing problem. A range of new penalized likelihood-type methods are proposed for obtaining the upper bound. These ideas inspired the follow-up paper \citep{gao2015sbmalgo} to develop polynomial-time and rate-optimal algorithms.

Theorem \ref{thm:main} covers both dense and sparse networks. It holds for a wide range of possible values of $a$ and $b$, from a constant order to an order of $n$. It implies that when the connectivity probability $\frac{a}{n}$ is $O(n^{-1})$, no consistent algorithm exists for community detection. The number of communities $K$ is allowed to grow fast. It can be as large as in the order of $n/\log n$ when the connectivity probability is a constant order, in which each community contains an order of $\log n$ nodes. In addition, for finite number of communities, Theorem \ref{thm:main} shows $\frac{(a-b)^2}{a}\rightarrow\infty$ is a necessary and sufficient condition for consistent community detection, which implies consistency results in \citep{mossel2012stochastic, mossel2013proof}. It also recovers the strong consistency results in \citep{mossel2014consistency, hajek2015achieving}, in which they additionally assume $a\asymp \log n$.

The minimax rate is of an exponential form, contrast to the polynomial rates in \citep{rohe2011spectral, lei2013consistency}. The term $\frac{nI}{K}$ plays a dominating role in determining the rate. Consider the $\beta=1$ case. Rewrite $\frac{nI}{K}$ in the form of $\rho\log n$, and then approximately we fail to recover essentially $n^{1-\rho}$ nodes. When $\rho>1$, the network enjoys strong consistency property (exact recovery) since $n^{1-\rho}=o(1)$, i.e., every node is correctly clustered. While for $0<\rho<1$, it is impossible to recover the communities exactly.
\newline
\\
\textbf{Organization.} The paper is organized as follows. The fundamental limits of community detection are discussed in Section \ref{sec:fundamental}. We present the penalized likelihood-type procedures in Section \ref{sec:procedure} to achieve the optimal rate. Some special cases of our result and the computational feasibility are discussed in Section \ref{sec:discussion}. Section \ref{sec:proof_theorem} gives the proofs of the main theorems, while Section \ref{sec:proof_lemma} provides the proofs of key technical lemmas.
\newline
\\
\textbf{Notation.} For any set $B$, we use $|B|$ to indicate its cardinality. For two arbitrary equal-length vectors $x=\{x_i\}$ and $y=\{y_i\}$, define the Hamming distance between $x$ and $y$ as $d_H(x,y)=|\{i:x_i\neq y_i\}|$, i.e., the number of coordinates with different values. For any positive integer $m$, we use $[m]$ to denote the set $\{1,2,\ldots,m\}$. For any two random variables $X$ and $Y$, we use $X\perp Y$ to indicate that they are independent. Denote $\Ber(q)$ as a Bernoulli distribution with success probability $q$, and $\Bin(m,q)$ as a binomial distribution with $m$ trials and success probability $q$. For two positive sequences ${x_n}$ and $y_n$, $x_n\lesssim y_n$ means $x_n\leq cy_n$ for some  constant $c$ not depending on $n$. We adopt the notation $x\asymp y$ if $x_n\lesssim y_n$  and $y_n\lesssim x_n$. For any scalar $z$, let $\lfloor z\rfloor=\{m\in\mathbb{Z}:m\leq z\}$ and $\lceil z\rceil=\{m\in\mathbb{Z}:m\geq z\}$. %we define $\lfloor z\rfloor$ to be largest integer not greater than $z$, and $\lceil z\rceil$ to be the smallest integer not less than $z$. 
We use $\Theta$ short for $\Theta(n,K,a,b,\beta)$ when there is no ambiguity to drop the index $(n,K,a,b,\beta)$.

\section{Fundamental Limits of Community Detection}\label{sec:fundamental}

\subsection{Mis-match Ratio}\label{subsec::mis_match}
Before giving the exact definition of mis-match ratio, we need to introduce permutations $\Delta:[K]\rightarrow[K]$ to define equivalent partitions. For the community detection problem, there exists an identifiability issue involved with the community label. For instance, for a network with 4 nodes, assignments $(1,1,2,2)$ and $(2,2,1,1)$ give the same network partition. Define $\delta\circ\sigma$ as $\delta\in\Delta$ to be a new assignment with $(\delta\circ\sigma)(i)=\delta(\sigma(i))$ for each $i\in[n]$. This assignment is equivalent to $\sigma$. The mis-match ratio is used as the loss function, counting the proportion of nodes incorrectly clustered, minimizing over all the possible permutations as follows,
\begin{align*}
r(\sigma, \hat\sigma)=\inf_{\delta}d_H(\sigma,\delta\circ\hat\sigma)/n.
\end{align*}
The Hamming distance between $\sigma$ and $\hat\sigma$ is just to count the number of entries having different values in two vectors. Thus $r(\sigma, \hat\sigma)$ is the total number of errors divided by the total number of nodes.

\subsection{Homogeneous Stochastic Block Model}\label{subsec:homogeneous}
The Stochastic Block Model assumes the network has an underlying blockwise structure. When all $\{\theta_{i,j}\}$ take two possible values $a/n$ or $b/n$, depending on whether $\sigma(i)=\sigma(j)$ or not, we call the SBM $\textit{homogeneous}$. In this case $\{\theta_{i,j}\}$ is unique for any given $\sigma$. The homogeneous SBM is the most studied model in computer science and probability \citep{mossel2012stochastic, mossel2013proof, mossel2014consistency, hajek2015achieving, chen2014statistical}. Define
\begin{align*}
\Theta_1(n,K,a,b,\beta)\triangleq\Big\{(\sigma,\{\theta_{i,j}\})\in\Theta(n&,K,a,b,\beta):\; \theta_{i,j}=  \frac{a}{n} \text{ if }\sigma(i)=\sigma(j)\\
&\text{ and }\theta_{i,j}=  \frac{b}{n} \text{ if }\sigma(i)\neq\sigma(j),\;\forall i\neq j\Big\}.
\end{align*} 
This is a homogeneous SBM. In $\Theta_1$, since $\{\theta_{i,j}\}$ is uniquely determined by any given $\sigma$, we may write $\sigma\in \Theta_1$ instead of $(\sigma,\{\theta_{i,j}\})\in\Theta_1$ for simplicity. The same rule may be applied for any other homogeneous SBM.

Note that $\Theta_1$ is \textit{closed under permutation}. Let $\pi$ be any permutation on $[n]$, then for any $\sigma\in\Theta_1$, a new assignment $\sigma'$ defined as $\sigma'(i)=\sigma(\pi^{-1}(i))$ also belongs to $\Theta_1$. This property is very helpful for us to show $\Theta_1$ is a least favorable subspace of $\Theta$ for community detection. A minimax lower bound over $\Theta_1$ immediately gives a lower bound for a larger parameter space, such as $\Theta$.

\subsection{From Global to Local}
To establish a lower bound is challenging to work with the loss function $r(\sigma,\hat\sigma)$ directly, as it takes infimum over an equivalent class. The mis-match ratio is a \textit{global} property of the network. The key idea in this paper is to define a \textit{local} loss, and to reduce the global minimax problem into a local classification for one node.

The local loss focuses only on one node. Given the truth $\sigma$ and any procedure $\hat\sigma$, the loss of estimating the label for the $i$-th node is defined as follows. Let $S_{\sigma}(\hat\sigma)=\{\sigma':\sigma'=\delta\circ\hat\sigma,\delta\in\Delta,\;d_H(\sigma',\sigma)=\inf_{\delta}d_H(\sigma,\delta\circ\hat\sigma)\}$, and define
\begin{align*}
r(\sigma(i),\hat\sigma(i))\triangleq \sum_{\sigma'\in S_{\sigma}(\hat\sigma)}\frac{d_H(\sigma(i),\sigma'(i))}{|S_{\sigma}(\hat\sigma)|},
\end{align*}
for each $i\in[n]$. It is an average over all the possible $\sigma'\in S_{\sigma}(\hat\sigma)$.

We will see later that it is relatively easy to study the local loss. Lemma \ref{lem:global_local} shows that the global loss is equal to the local one when the SBM is homogeneous and closed under permutation.

\begin{lemma}[Global to local]\label{lem:global_local}
Let $\Lambda$ be any homogeneous parameter space that is closed under permutation. Let $\tau$ be the uniform prior over all the elements in $\Lambda$. Define the global Bayesian risk as $B_\tau(\hat\sigma)=\frac{1}{|\Lambda|}\sum_{\sigma\in\Lambda}\E r(\sigma,\hat\sigma)$ and the local Bayesian risk $B_\tau(\hat\sigma(1))=\frac{1}{|\Lambda|}\sum_{\sigma\in\Lambda}\E r(\sigma(1),\hat\sigma(1))$ for the first node. Then
\begin{align*}
\inf_{\hat\sigma}B_\tau(\hat\sigma)=\inf_{\hat\sigma}B_\tau(\hat\sigma(1)).
\end{align*}
\end{lemma}

The proof of Lemma \ref{lem:global_local} is involved. It is established by exploiting the property of exchangeability of the parameter space $\Lambda$.

\subsection{Minimax Lower Bound}
By constructing a least favorable case of $\Theta_1$, we have the following lower bound for the minimax rate. We present the lower bound under milder conditions than what is stated in Theorem \ref{thm:main}.

\begin{theorem}\label{thm:lower_beta}
Under the assumption $\frac{nI}{K}\rightarrow\infty$, we have
\begin{align}\label{eqn:lower_beta}
\inf_{\hat\sigma}\sup_{\Theta_1(n,K,a,b,\beta)}\mathbb{E}r(\sigma,\hat\sigma)\geq
\begin{cases}
\exp\big(-(1+o(1))\frac{nI}{2}\big),K=2,\\
\exp\big(-(1+o(1))\frac{nI}{\beta K}\big),K\geq 3.
\end{cases}
\end{align}
If $\frac{nI}{K}=O(1)$, then $\inf_{\hat\sigma}\sup_{\Theta_1(n,K,a,b,\beta)}\R(\sigma,\hat\sigma)\geq c$ for some positive constant $c$.
\end{theorem}

The forms of minimax rates are different for two cases $K\geq 3$ and $K=2$. For $K\geq 3$, it is relatively more challenging to discover and distinguish small communities, rather than the communities with larger sizes. The least favorable case is the case for which at least a constant proportion of communities are of size $\frac{n}{\beta K}$. %The nodes in these small communities are the most difficult nodes to be correctly recovered. 
The hardness of the community detection in this setting is then determined by the ability to recover and distinguish such small communities. For $K=2$, the least favorable setting in $\Theta_1$ is when the two communities are of the same size. When there are only two communities, it is actually easier to recover the non-equal-sized communities, by identifying the larger one first and then labeling the remaining nodes as from the smaller one.
~\\
\newline
\noindent\textbf{Approximately Equal-Sized Case}: We are interested in the case with $\beta=1+o(1)$, where communities are almost of the same size. Networks of community sizes exactly equal to $\frac{n}{K}$ are the most studied settings \citep{chen2014statistical, mossel2013proof, chin2015stochastic}. Here we allow a small fluctuation of community sizes. Denote $\Theta^0$ as follows,
\begin{align*}
\Theta^0(&n,K,a,b)\triangleq\Bigg\{(\sigma,\{\theta_{i,j}\}):\;\sigma:[n]\rightarrow [K]^n,n_k=(1+o(1))\frac{nI}{K},\forall k\in[K],\\
&\theta_{i,i}=0,\;\forall i\in[n],\;\theta_{i,j}= \frac{a}{n} \text{ if }\sigma(i)=\sigma(j)\text{ and }\theta_{i,j}=  \frac{b}{n} \text{ if }\sigma(i)\neq\sigma(j),\;\forall i\neq j\Bigg\}.
\end{align*}
Note that $\Theta^0(n,K,a,b)$ is $\Theta_1(n, K, a, b, \beta)$ with $\beta=1+o(1)$, for which we have the following minimax lower bound.

\begin{theorem}\label{thm:lower}
Under the assumption $\frac{nI}{K}\rightarrow\infty$, we have
\begin{align}\label{eqn:lower}
\inf_{\hat\sigma}\sup_{\Theta^0(n,K,a,b)}\E r(\sigma,\hat\sigma)\geq \exp\Big(-(1+o(1))\frac{nI}{K}\Big).
\end{align}
If $\frac{nI}{K}=O(1)$, then $\inf_{\hat\sigma}\sup_{\Theta^0(n,K,a,b)}\E r(\sigma,\hat\sigma)\geq c$ for some positive constant $c$.
\end{theorem}

\noindent Compared with Theorem \ref{thm:lower_beta}, the forms of rates for $K=2$ and $K\geq 3$ are the same in $\Theta^0$. The proof of Theorem \ref{thm:lower} is provided in Section \ref{sec:proof_theorem}. We defer the proof of Theorem \ref{thm:lower_beta} to the supplement material \citep{supplement}, since it is almost identical to that of Theorem \ref{thm:lower}.

\section{Rate-optimal Procedure}\label{sec:procedure}
We develop a range of penalized likelihood-type procedures to achieve the optimal mis-match ratio. 
Throughout the section $\sigma_0$ is denoted as the underlying truth. 

\subsection{Penalized Likelihood-type Estimation}
The penalized procedure is based on the likelihood of a homogeneous network, although risk upper bounds are established for more general networks. If the network is homogeneous ($\Theta^0$ and $\Theta_1$), for which the within and between community probabilities are exactly equal to $a/n$ and $b/n$ respectively, the log-likelihood function is
\begin{align*}
L(\sigma;A)&=\log(\frac{a}{n})\sum_{i<j}A_{i,j}1_{\{\sigma(i)=\sigma(j)\}}  +\log (1-\frac{a}{n})\sum_{i<j}(1-A_{i,j})1_{\{\sigma(i)=\sigma(j)\}}\\
&+\log(\frac{b}{n})\sum_{i<j}A_{i,j}1_{\{\sigma(i)\neq\sigma(j)\}}  +\log (1-\frac{b}{n})\sum_{i<j}(1-A_{i,j})1_{\{\sigma(i)\neq\sigma(j)\}}.
\end{align*}
Since $\sum_{i<j}A_{i,j}1_{\{\sigma(i)=\sigma(j)\}}+\sum_{i<j}A_{i,j}1_{\{\sigma(i)\neq\sigma(j)\}}=\sum_{i<j}A_{i,j}$ for all $\sigma$, we can write $L(\sigma;A)$ as
\begin{align*}
L(\sigma;A)=\log\frac{a(1-b/n)}{b(1-a/n)}\sum_{i<j}A_{i,j}1_{\{\sigma(i)=\sigma(j)\}}-\log\frac{1-b/n}{1-a/n}\sum_{i<j}1_{\{\sigma(i)=\sigma(j)\}}+f(A),
\end{align*}
where $f(A)$ is a function not depending on $\sigma$.
Then the maximum likelihood estimator $\hat\sigma^{MLE}$ is as follows,
\begin{align}\label{eqn:mle}
\begin{split}
\hat\sigma^{MLE}&=\argmax_{\sigma}L(\sigma;A)\\
&=\argmax_{\sigma} \log\frac{a(1-b/n)}{b(1-a/n)}\sum_{i<j}A_{i,j}1_{\{\sigma(i)=\sigma(j)\}}-\log\frac{1-b/n}{1-a/n}\sum_{i<j}1_{\{\sigma(i)=\sigma(j)\}}.
\end{split}
\end{align}
The above maximum likelihood estimator can be decomposed into two terms. The first one is the sum of all $A_{i,j}$ for all $i$ and $j$ belonging to the same communities of $\sigma$. The second term is a penalty over the sum of sizes of all communities. There is a trade-off between these two terms. The first term is maximized when there is only one community, while the second term, a penalty term, is maximized when all community sizes are equal. However the second term is dropped when the community sizes are required to be exactly equal, i.e., the maximum likelihood estimator over all $\sigma$ with a community size $n/K$ for every community has a simpler form, $\hat\sigma^{MLE}=\argmax_\sigma\sum_{i<j}A_{i,j}1_{\{\sigma(i)=\sigma(j)\}}$. 

When the parameter space is not homogeneous (e.g. $\Theta$), the maximum likelihood estimator may not have a simple form as Equation (\ref{eqn:mle}). However, we still propose to use the identical simple form of penalized likelihood estimator as Equation (\ref{eqn:mle}), i.e., 
\begin{align*}
\hat\sigma=\argmax_{\sigma\in\Theta}T(\sigma)\text{ with }T(\sigma)\triangleq \sum_{i<j}A_{i,j}1_{\{\sigma(i)=\sigma(j)\}}-\lambda\sum_{i<j}1_{\{\sigma(i)=\sigma(j)\}},
\end{align*}
where we set
\begin{align}
\lambda=\log\Big(\frac{1-b/n}{1-a/n}\Big)\Big/\log\Big(\frac{a(1-b/n)}{b(1-a/n)}\Big),\;\forall K\geq 2.\label{eqn:lambda}
\end{align}
When the parameter space is homogeneous, $\hat\sigma$ is identical to the maximum likelihood estimator. The optimality result will be obtained for the parameter space $\Theta$, which allows the network to be inhomogeneous, and imbalanced in  the sense that the community sizes may be different.

\subsection{Other Choices of $\lambda$}
In the previous section we provide a unified $\lambda$ for the penalized likelihood-type estimation for both $K=2$ and $K\geq 3$. It is worthwhile to point out that for $K\geq 3$ the optimality can be attained for a wide range of $\lambda$. Let $t^{\star}=\frac{1}{2}\log\frac{a(1-b/n)}{b(1-a/n)}$. It can be shown that $t^{\star}$ is the minimizer of the moment generating function for the difference of two Bernoulli variables, i.e., $t^\star=\argmin_{t>0} \E e^{t(X-Y)}$, where $X\sim\text{Ber}(\frac{b}{n})$ and $Y\sim\text{Ber}(\frac{a}{n})$. It is equivalent to write $\lambda$ in Equation (\ref{eqn:lambda}) as follows,
\begin{align*}
\lambda&=-\frac{1}{2t^\star}\log\Big(\frac{\frac{a}{n}\exp(-t^\star)+1-\frac{a}{n}}{\frac{b}{n}\exp(t^\star)+1-\frac{b}{n}}\Big)\\
&=-\frac{1}{2t^\star}\log\Big(\frac{a}{n}\exp(-t^\star)+1-\frac{a}{n}\Big)+\frac{1}{2t^\star}\log\Big(\frac{b}{n}\exp(t^\star)+1-\frac{b}{n}\Big).
\end{align*}
From the equation above, we can interpret $\lambda$ as a weighted sum between two terms, with the first one more involving the within-community probability $\frac{a}{n}$, and the second more focusing on the between-community probability $\frac{b}{n}$. Define
\begin{align}
\lambda=
\begin{cases}
-\frac{1}{2t^\star}\log\big(\frac{a}{n}e^{-t^\star}+1-\frac{a}{n}\big)+\frac{1}{2t^\star}\log\big(\frac{b}{n}e^{t^\star}+1-\frac{b}{n}\big),\;K=2\\
-\frac{w}{t^\star}\log\big(\frac{a}{n}e^{-t^\star}+1-\frac{a}{n}\big)+\frac{(1-w)}{t^\star}\log\big(\frac{b}{n}e^{t^\star}+1-\frac{b}{n}\big)\;K\geq 3,
\end{cases}\label{eqn:lambda2}
\end{align}
where $w$ in any constant in $[0,1]$. We can clearly see that $\lambda$ in Equation (\ref{eqn:lambda}) is a special case of $\lambda$ in (\ref{eqn:lambda2}) with $w=1/2$. In Section \ref{subsec:upper}, we give theoretical properties of penalized likelihood estimation for all $\lambda$ in Equation (\ref{eqn:lambda2}).

\subsection{Minimax Upper Bound}\label{subsec:upper}
For the general SBM $\Theta$, the risk upper bound of the penalized likelihood estimator, for every $\lambda$ in Equation (\ref{eqn:lambda2}), defined in the previous section, matches the minimax lower bound given in Theorem \ref{thm:lower_beta}.
\begin{theorem}\label{thm:upper_beta}
Assume $\frac{nI}{K\log K}\rightarrow\infty$ and $K\geq 2$. For the penalized maximum likelihood estimator $\hat\sigma$ with $\lambda$ defined in (\ref{eqn:lambda2}), we have
\begin{align*}
\sup_{\Theta(n,K,a,b,\beta)}\mathbb{E}r(\hat{\sigma},\sigma)\leq
\begin{cases}
\exp\big(-(1+o(1))\frac{nI}{2}\big),K=2,\\
\exp\big(-(1+o(1))\frac{nI}{\beta K}\big),K\geq 3,
\end{cases}
\end{align*}
where $1\leq \beta<\sqrt{5/3}$.
\end{theorem}
\noindent\textbf{Approximately Equal-Sized Case}: For the special parameter space $\Theta^0$ for which community sizes are almost equal, we have the following result, a form analogous to Theorem \ref{thm:upper_beta}.

\begin{theorem}\label{thm:upper}
Assume $\frac{nI}{K\log K}\rightarrow\infty$ and $K\geq 2$. For the penalized maximum likelihood estimator $\hat\sigma$ with $\lambda$ defined in (\ref{eqn:lambda2}), we have
\begin{align*}
\sup_{\Theta^0(n,K,a,b)}\mathbb{E}r(\hat{\sigma},\sigma)\leq\exp\big(-(1+o(1))\frac{nI}{K}\big).
\end{align*}
\end{theorem}
\noindent The proof of the above theorem is provided in Section \ref{sec:proof_theorem}. Due to the similarity, the proof of Theorem \ref{thm:upper_beta} is given in the supplement material \cite{supplement}.

\section{Discussion}\label{sec:discussion}
\subsection{Implications on Sharp Thresholds}
The minimax rates in Theorem \ref{thm:main} immediately imply various sharp thresholds in \citep{mossel2012stochastic, mossel2013proof, mossel2014consistency, hajek2015achieving}. By letting the rates equal to $o(1/n)$ or $o(1)$, we can get critical values for strong and weak consistency respectively, under various settings.
\\
\newline
\textbf{Special Case with $a=o(n)$ and $\frac{a-b}{a}=o(1)$.}
Under this scenario the difference of within-community probability and between-community probability is relatively small. Note that $I=(1+o(1))(a-b)^2/(4an)$, which reduces the minimax result into the form of $\exp(-(1+o(1))(a-b)^2/(4aK))$. In the case of $K=2$, Theorem \ref{thm:main} implies the results from \citep{mossel2012stochastic, mossel2013proof}. With the additional assumption $a,b=n^{o(1/\log\log n)}$, they show that $(a-b)^2/a\rightarrow\infty$ is the necessary and sufficient condition to get consistency. It also agrees with the sharp threshold for strong consistency in \citep{mossel2014consistency}.
\\
\newline
\textbf{Special Case with Probability in the Order of $\log n$.} Consider a more special setting where $a$ and $b$ are in the order of $\log n$. Denote $a=e_1\log n$ and $b=e_2\log n$, with $e_1\geq e_2>0$. Note that $I$ can be written as $I=(1+o(1))(\sqrt{e_1}-\sqrt{e_2})^2\log n/n$. 
\begin{corollary}
Assume $K=n^{o(1)}$. There exists a strongly consistent estimator if $\liminf_{n\rightarrow\infty}\frac{\sqrt{e_1}-\sqrt{e_2}}{\sqrt{K}}> 1$.
\end{corollary}
\noindent For any finite $K$, the recovery threshold is identical to the result in \cite{hajek2015achieving}. For the two-community case with $e_1$ and $e_2$ constants, $\sqrt{e_1}-\sqrt{e_2}>\sqrt{2}$ for exact recovery is proved in \citep{mossel2014consistency}.

\subsection{Computational Feasibility}
The penalized likelihood estimator we propose searches all the possible assignments in the parameter space. It is computationally intractable due to the enormous cardinality of the assignments. However, the idea of reducing global estimation into local testing problem we developed in this paper establishes a guideline for constructing both efficient and optimal algorithms. Along with the \textit{global to local} scheme, the penalized likelihood estimator can be further modified into an node-wise procedure, whose purpose is to assign the label node by node. In this way the exhaustive search over the parameter space is avoided and the computation complexity is dramatically reduced. By exploiting the local idea, in the subsequent paper \citep{gao2015sbmalgo} a two-stage algorithm is proposed to simultaneously achieve the optimal rate and computational feasibility.

\section{Proofs of Main Theorems}\label{sec:proof_theorem}
In this section, we prove two main theorems, Theorem \ref{thm:lower} and Theorem \ref{thm:upper}. The proofs of Theorem \ref{thm:lower_beta} and Theorem \ref{thm:upper_beta} are almost identical to those of Theorem \ref{thm:lower} and Theorem \ref{thm:upper}. We put them in the supplement material \cite{supplement}.

\subsection{Proof of Theorem \ref{thm:lower}}
To get the lower bound for the parameter space $\Theta^0$, we will first construct and analyze a least favorable case in term of the sizes of the communities. In particular the community sizes only take value in $\{\lfloor \frac{n}{K}\rfloor, \lfloor \frac{n}{K}\rfloor+1, \lfloor \frac{n}{K}\rfloor-1\}$, and the number of communities with size $\lfloor\frac{n}{K}\rfloor$ or $\lfloor\frac{n}{K}\rfloor+1$ is of a constant proportion of $K$.

First consider the case with $K\geq 3$. For each pair of $(n,K)$, the integer $K$ can always be decomposed as the sum of three integers: $K=K_1+K_2+K_3$, satisfying \textit{(1)} there exists a constant $\epsilon>0$ such that $\epsilon K<\min(K_1,K_2)\leq\max(K_1, K_2)<(1-\epsilon)K$; and \textit{(2)} either of the following two conditions:
\begin{align}
&\lfloor \frac{n}{K}\rfloor K_1 + (\lfloor \frac{n}{K}\rfloor+1) K_2 + (\lfloor \frac{n}{K}\rfloor-1)K_3 = n;\label{eqn:K1}\\
\text{or }&\lceil \frac{n}{K}\rceil K_1 + (\lceil \frac{n}{K}\rceil+1) K_2 + (\lceil \frac{n}{K}\rceil-1)K_3 = n;\label{eqn:K2}
\end{align}
When $K\geq 3$, it can be shown that such decomposition always exists. Write $n=\lfloor \frac{n}{K}\rfloor K +r$, where $0\leq r\leq K-1$ is an integer. If $r\geq 2\epsilon K$ and $r\leq (1-2\epsilon) K$ for a constant $\epsilon>0$, we have $n=\lfloor \frac{n}{K}\rfloor (K-r)+ (\lfloor \frac{n}{K}\rfloor+1)r$, which satisfies Equation (\ref{eqn:K1}). Otherwise, if $r<2\epsilon K$ for a small positive constant $\epsilon$, write $n=\lfloor \frac{n}{K}\rfloor (K-2\lfloor\frac{K}{3}\rfloor-r)+(\lfloor \frac{n}{K}\rfloor+1) (\lfloor\frac{K}{3}\rfloor+r) +(\lfloor \frac{n}{K}\rfloor -1) \lfloor\frac{K}{3}\rfloor$, which satisfies Equation (\ref{eqn:K1}) for $\epsilon$ sufficient small. If $K-r>2\epsilon K$, we may argue similarly to get Equation (\ref{eqn:K2}).

Recall that we use $n_k$ to denote the size of the $k$-th community for each $k\in[K]$. Without loss of generality, assume there exist $\{K_i\}_{1\leq i \leq 3}$ satisfying Equation (\ref{eqn:K1}) with $\epsilon K<\min(K_1,K_2)\leq\max(K_1, K_2)<(1-\epsilon)K$. Define a subparameter space of $\Theta^0$ as follows,
\begin{align*}
\Theta^{L}(n,K,a,b,\{K_i\})&=\Big\{(\sigma,\{\theta_{i,j}\})\in\Theta^0(n,K,a,b):\big|\{k:n_k=\lfloor \frac{n}{K}\rfloor\}\big|=K_1,\\
&\big|\{k:n_k=\lfloor \frac{n}{K}\rfloor+1\}\big|=K_2, \big|\{k:n_k=\lfloor \frac{n}{K}\rfloor -1\}\big|=K_3\Big\}.
\end{align*}

For the case with $K=2$, we can define the least favorable case in an analogous way. It has a slight different form depending on whether $n/2$ is an integer or not. If $\frac{n}{2}\neq\lfloor\frac{n}{2}\rfloor$, $\Theta^L(n,2,a,b)\triangleq\big\{(\sigma,\{\theta_{i,j}\})\in\Theta^0(n,2,a,b):(n_1,n_2)=(\lfloor\frac{n}{2}\rfloor, \lceil\frac{n}{2}\rceil)\big\}$. Otherwise, $\Theta^L(n,2,a,b)\triangleq\big\{(\sigma,\{\theta_{i,j}\})\in\Theta^0(n,2,a,b):(n_1,n_2)\in\{(\frac{n}{2}, \frac{n}{2}), (\frac{n}{2}+1, \frac{n}{2}-1)\}\big\}$.

Note that $\Theta^L$ is homogeneous and closed under permutation. Compared with $\Theta^0$, $\Theta^L$ is quite small, enough for us to do some lower bound analysis. On the other hand, it is large enough to match the lower bound in Equation (\ref{eqn:lower_beta}).

\begin{lemma}\label{lem::multiple_test}
Let $\tau$ be the uniform prior over all the elements in $\Theta^L$. For the first node, define the local Bayesian risk to be $B_\tau(\hat\sigma(1))=\frac{1}{|\Theta^L|}\sum_{\sigma\in\Theta^L}\E r(\sigma(1),\hat\sigma(1))$. Then there exists a constant $\epsilon>0$ such that
\begin{align*}
B_\tau(\hat\sigma(1))\geq \epsilon\p\Big(\sum_{u=1}^{\lfloor n/K\rfloor}X_u\geq \sum_{u=1}^{\lfloor n/K\rfloor}Y_u\Big),
\end{align*}
where $X_i\stackrel{iid}{\sim}\Ber(\frac{b}{n})$, $Y_i\stackrel{iid}{\sim}\Ber(\frac{a}{n})$, for $i=1,2,\ldots,\lfloor\frac{n}{K}\rfloor$, and $\{X_i\}_{i=1}^{\lfloor\frac{n}{K}\rfloor}\perp\{Y_i\}_{i=1}^{\lfloor\frac{n}{K}\rfloor}$.
\end{lemma}

Lemma \ref{lem::multiple_test} shows the lower bound is only involved with $2\lfloor \frac{n}{K}\rfloor$ Bernoulli random variables, whose success probability is either $a/n$ or $b/n$. Recall that $a/n$ is the smallest within-community probability and $b/n$ is the largest between-community probability. The lower bound here will be determined by testing two probability measures. In $\Theta^L$, the most difficult case is testing two assignment vectors with Hamming distance 1. The difference of their probability measures is exactly the difference between probability measures of $X$ and $Y$.

\begin{lemma}\label{lem::binomial_tail}
Let $n'=\lfloor \frac{n}{K}\rfloor$. Define $Z_i=X_i - Y_i$ with $\{X_{i}\}\stackrel{iid}{\sim}\Ber( \frac{b}{n} ),\;\{Y_{i}\}\stackrel{iid}{\sim}\Ber( \frac{a}{n} )$, and $\{X_i\}\perp\{Y_i\}$, for $i=1,2,\ldots,n'$. If $\frac{nI}{K}\rightarrow\infty$, we have
\begin{align*}
\p\left(\frac{1}{n'}\sum_{i=1}^{n'}Z_{i}>0\right)\geq \exp\left(-(1+o(1))nI/K\right).
\end{align*}
In addition, if $nI/K=O(1)$, then $\p\left(\frac{1}{n'}\sum_{i=1}^{n'}Z_{i}>0\right)\geq c$ for some positive constant $c>0$.
\end{lemma}

Lemma \ref{lem::binomial_tail} provides an explicit expression for the lower bound. The proof mainly follows the proof of Cramer-Chernoff Theorem \cite{van2000asymptotic}. The general Cramer-Chernoff Theorem gives a lower bound for the tail probability that the sum of random variables deviates from its mean. Usually it is for the case where these random variables are from a distribution independent of the sample size. In our setting we allow $a$ and $b$ to depend on $n'$.

\begin{proof}[Proof of Theorem \ref{thm:lower}]
Since $\Theta^L\subset\Theta^0$, we have $\inf_{\hat\sigma}\sup_{\Theta^0}\E r(\sigma,\hat\sigma)\geq\inf_{\hat\sigma}\sup_{\sigma\in\Theta^L}\E r(\sigma,\hat\sigma)$. Due to the fact that Bayes risk always lower bounds the global risk we have $\inf_{\hat\sigma}\sup_{\sigma\in\Theta^L}\E r(\sigma,\hat\sigma)\geq \inf_{\hat\sigma}\sup_{\sigma\in\Theta^L}B_\tau(\hat\sigma)$. By the fact that $\Theta^L$ is a homogeneous parameter space closed under permutation for both $K\geq 3$ and $K=2$, Lemma \ref{lem:global_local} implies $\inf_{\hat\sigma}\sup_{\sigma\in\Theta^L}B_\tau(\hat\sigma)=\inf_{\hat\sigma}\sup_{\sigma\in\Theta^L}B_\tau(\hat\sigma(1))$. Thus
\begin{align*}
\inf_{\hat\sigma}\sup_{\Theta^0}\E r(\sigma,\hat\sigma)\geq\inf_{\hat\sigma}\sup_{\sigma\in\Theta^L}B_\tau(\hat\sigma(1)),
\end{align*}
which, together with Lemma \ref{lem::multiple_test} and Lemma \ref{lem::binomial_tail}, implies Equation (\ref{eqn:lower}) of Theorem \ref{thm:lower}.
\end{proof}

\subsection{Proof of Theorem \ref{thm:upper}}\label{subsec:proof_upper}
Recall that $\Delta$ is the set of all permutations from $[K]$ to $[K]$. For an arbitrary $\sigma\in\Theta^0$, define $\Gamma(\sigma)$ as the equivalent class of $\sigma$ with $\Gamma(\sigma)=\{\sigma':\exists\delta\in\Delta\text{, s.t. }\sigma'=\delta\circ\sigma\}$.
We use the notation $\Gamma$ as a general reference for equivalent class, and $\{\Gamma\}$ as the set consisting of all the possible equivalent classes with respect to $\Theta^0$. For any $\sigma_{1},\sigma_{2}\in\Theta$, define the distance between $\sigma_{1}$ and $\sigma_{2}$ as
\begin{align*}
d(\sigma_{1},\sigma_{2})\triangleq\inf_{\sigma'_{2}\in\Gamma(\sigma_{2})}d_H(\sigma_{1},\sigma'_{2})=\inf_{\sigma'_{1}\in\Gamma(\sigma_{1}),\sigma'_{2}\in\Gamma(\sigma_{2})}d_H(\sigma'_{1},\sigma'_{2}).
\end{align*}
Here we view $d(\cdot,\cdot)$ as a distance between the equivalent class $\Gamma(\sigma_1)$ and $\Gamma(\sigma_2)$. Accordingly the mis-match ratio $r(\sigma,\hat\sigma)$ is exactly equal to
\begin{align*}
r(\sigma,\hat\sigma)=\frac{1}{n}d(\sigma,\hat\sigma).
\end{align*}

In the following sections we denote the true assignment by $\sigma_0$. Define 
\begin{align}
P_m=\p\big(\exists \sigma\in \Theta^0: d(\sigma_0,\sigma)= m\text{ and }T(\sigma)\geq T(\sigma_0)\big)\label{eqn:P}
\end{align}
for any integer $m$ with $0<m<n$. The key step is to get a tight bound of the probability $\p \big(T(\sigma)\geq T(\sigma_{0})\big)$ for one fixed assignment $\sigma$ satisfying $d(\sigma,\sigma_{0})=m$. Let $\{n_k\}$ to be the size of communities under the truth $\sigma_0$. Without loss of generality, assume $\sigma_0(i)=k$ for any $i\in[\sum_{j\leq k-1}n_j+1,\sum_{j\leq k}n_j]$. Then the value of $2\sum_{i<j}A_{i,j}{\{\sigma_0(i)=\sigma_0(j)\}}$ is just to add up all the entries in the $K$ diagonal blocks of the adjacency matrix $A$. It is illustrated by color plates in the Figure \ref{fig_color_plate}. The gray parts represent the within-community connections, and blank parts represent the between-community connections. It is obvious to see that $2\sum_{i<j}A_{i,j}{\{\sigma_0(i)=\sigma_0(j)\}}$ precisely includes all the gray parts, i.e., all the Bernoulli random variables with success probability $\frac{a}{n}$ in the adjacency matrix.

\begin{figure}[ht]\centering
\begin{minipage}{0.49\textwidth}
\includegraphics[trim = 35mm 25mm 25mm 20mm, clip, width=0.95\textwidth]{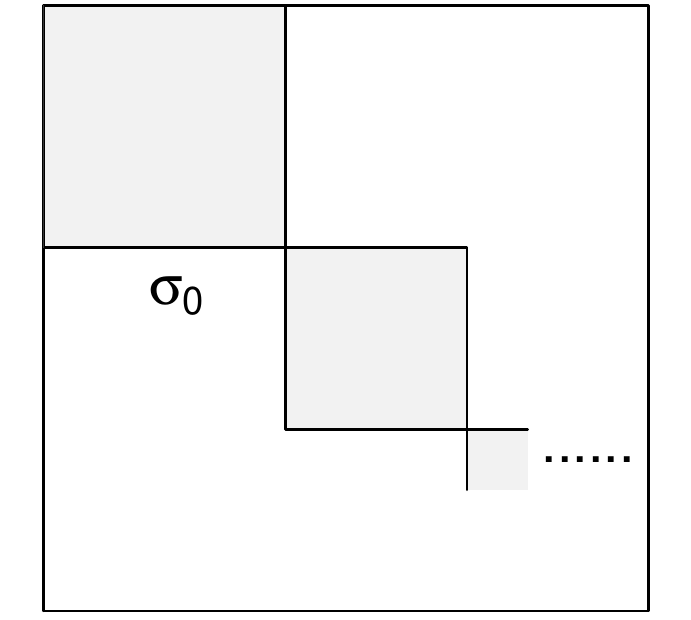}
\end{minipage}
\begin{minipage}{0.49\textwidth}
\includegraphics[trim = 35mm 25mm 25mm 20mm, clip, width=0.95\textwidth]{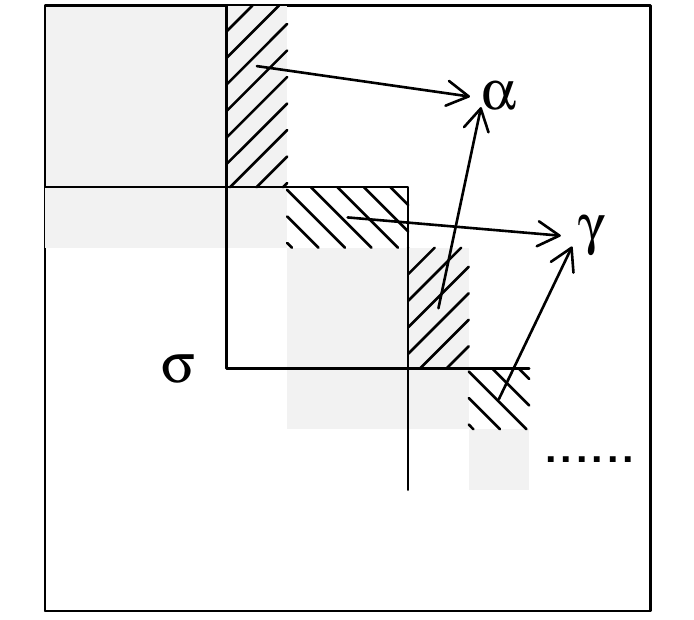}
\end{minipage}
\end{figure}
\begin{figure}[ht]
\caption{Each gray block stands for all the within-community connnection in one single community. The areas inside the squares are all the $A_{i,j}$ entries summed up. \textbf{Left:} For $2\sum_{i<j}A_{i,j}{\{\sigma_0(i)=\sigma_0(j)\}}$, the squares exactly overlap with the gray regions. \textbf{Right:} For $2\sum_{i<j}A_{i,j}{\{\sigma(i)=\sigma(j)\}}$, there would be some differences between the squares and gray parts, which are labeled as $\alpha$ or $\gamma$ according to their relative positions.} \label{fig_color_plate}
\end{figure}

When $d_H(\sigma,\sigma_0)=d(\sigma,\sigma_0)=m$, by comparing the two color plates in Figure \ref{fig_color_plate}, we can clearly see where the difference $\sum_{i<j}A_{i,j}{\{\sigma_0(i)=\sigma_0(j)\}}-\sum_{i<j}A_{i,j}{\{\sigma(i)=\sigma(j)\}}$ lies in. Note that
\begin{align*}
\sum_{i<j}A_{i,j}1_{\{\sigma(i)=\sigma(j)\}}-\sum_{i<j}A_{i,j}1_{\{\sigma_0(i)=\sigma_0(j)\}}&=\sum_{i<j}A_{i,j}1_{\{\sigma(i)=\sigma(j)\}}1_{\{\sigma_0(i)\neq \sigma_0(j)\}}\\&-\sum_{i<j}A_{i,j}1_{\{\sigma(i)\neq\sigma(j)\}}1_{\{\sigma_0(i)=\sigma_0(j)\}}.
\end{align*}
Define $\alpha(\sigma;\sigma_0)=|\{(i,j):i<j,\;\sigma_0(i)=\sigma_0(j)\text{ and }\sigma(i)\neq\sigma(j)\}|$, and $\gamma(\sigma;\sigma_0)=|\{(i,j):i<j,\;\sigma_0(i)\neq\sigma_0(j)\text{ and }\sigma(i)=\sigma(j)\}|$. We use the notations $\alpha$ and $\gamma$ for short when there is no ambiguity, then
\begin{align*}
\sum_{i<j}1_{\{\sigma(i)=\sigma(j)\}}-\sum_{i<j}1_{\{\sigma_0(i)=\sigma_0(j)\}}=\alpha-\gamma.
\end{align*}
The following proposition is helpful to study $P_m$ defined in Equation (\ref{eqn:P}).

\begin{proposition}\label{prop::chernoff}
Let $\sigma\in\Theta^0$ be an arbitrary assignment satisfying $d(\sigma,\sigma_{0})=m$, where $0<m<n$ is a positive integer. Then
\begin{align*}
\p \big(T(\sigma)\geq T(\sigma_{0})\big)&\leq\p\Big(\sum_{i=1}^{\gamma}X_{i}-\sum_{i=1}^{\alpha}Y_{i}\geq \lambda(\gamma-\alpha)\bigg|X_{i}\stackrel{iid}{\sim} \text{Ber}( \frac{b}{n} ),Y_{i}\stackrel{iid}{\sim} \text{Ber}( \frac{a}{n} )\Big)\\
&\leq \exp\left(-(\alpha\wedge\gamma) I\right),
\end{align*}
for $\lambda$ defined in Equation (\ref{eqn:lambda2}).\end{proposition}

Note that the value of $\gamma$ depends on $\sigma$ and $\sigma_0$. Lemma \ref{lem::number_black_white} provides a lower bound on $\gamma$ for each $m$.

\begin{lemma}\label{lem::number_black_white}
Let $\sigma\in\Theta^0$ be an arbitrary assignment satisfying $d(\sigma,\sigma_{0})=m$, where $0<m<n$ is a positive integer. Then 
\begin{align*}
\alpha(\sigma;\sigma_0)\wedge\gamma(\sigma;\sigma_0)\geq
\begin{cases}
\frac{(1-\eta)nm}{K}-m^2, \text{ if }m\leq \frac{n}{2K},\\
\frac{2(1-\eta)nm}{9K}, \text{ if }m> \frac{n}{2K}.
\end{cases}
\end{align*}
\end{lemma}

Lemma \ref{lem::number_black_white}, together with Proposition \ref{prop::chernoff}, immediately implies an upper bound on $\p \big(T(\sigma)\geq T(\sigma_{0})\big)$ for each given $\sigma$.

\begin{lemma}\label{lem::probability_likelihood}
Let $\sigma\in\Theta^0$ be an arbitrary assignment satisfying $d(\sigma,\sigma_{0})=m$, where $0<m<n$ is a positive integer. There exists a positive sequence $\eta\rightarrow 0$, independent of the choice of $\sigma$, such that
\begin{align*}
\p \big(T(\sigma)\geq T(\sigma_{0})\big)\leq
\begin{cases}
\exp\big(-\frac{(1-\eta)nmI}{K}+m^2I\big)\text{, if }m\leq \frac{n}{2K},\\
\exp\big(-\frac{2(1-\eta)nmI}{9K}\big)\text{, if }m\geq \frac{n}{2K},
\end{cases}
\end{align*}
for $\lambda$ defined in Equation (\ref{eqn:lambda2}).
\end{lemma}

We will apply a union bound to get an upper bound for $P_{m}$. It is worthwhile to point out that, in the union bound we should not use the cardinality of $\{\sigma\in\Theta^0: d(\sigma,\sigma_0)=m\}$, which is too large due to counting the assignments from the same equivalent class repetitively. Proposition \ref{prop::cardinality} gives an upper bound for cardinality of the equivalent class $\{\Gamma\}$.

\begin{proposition}\label{prop::cardinality}
The cardinality of equivalent class that has distance $m$ from $\sigma_{0}$ is upper bounded as follows,
\begin{equation*}
\Big|\Big\{\Gamma:\exists \sigma\in\Gamma\text{ s.t. }d(\sigma,\sigma_{0})=m\Big\}\Big|\leq \min\Big\{\big(\frac{enK}{m}\big)^{m},K^{n}\Big\},
\end{equation*}
where $0<m<n$ is a positive integer.

\end{proposition}

With Proposition \ref{prop::cardinality} and the union bound we are able to get a satisfactory bound by
\begin{align*}
P_{m}\leq \Big|\Big\{\Gamma:\exists \sigma\in\Gamma\text{ s.t. }d(\sigma,\sigma_{0})=m\Big\}\Big|\max_{\{\sigma:\;d(\sigma,\sigma_0)=m\}}\p \big(T(\sigma)\geq T(\sigma_{0})\big).
\end{align*}

\begin{proof}[Proof of Theorem \ref{thm:upper}] We only prove the case with $K\rightarrow\infty$ and $\frac{nI}{K\log K}\rightarrow\infty$. Let $\eta\rightarrow 0$ be a universal positive sequence given in  Lemma \ref{lem::probability_likelihood}. We consider three scenarios as follows.
\newline
~\\
(1) If $\liminf_{n\rightarrow\infty}\frac{nI}{K\log n}>1$, there exists a small constant $\epsilon>0$ such that $\frac{(1-\eta)nI}{K\log n}>1+\epsilon$. Let $\eta$ decay slowly such that both $\frac{\eta nI}{K\log K}$ and $\frac{\eta n}{K}$ go to infinity. We have $P_1\leq nK\exp\big(-\big(\frac{(1-\eta)n}{K}-1\big)I\big)\leq R$, where $R\triangleq n\exp\big(-(1-2\eta)nI/K\big)$. Since
\begin{align*}
n\R(\sigma,\hat\sigma)&\leq P_1 +\sum_{m=2}^{n}mP_m,
\end{align*}
it is sufficient to show $\sum_{i=2}^{n}mP_m$ is negligible compared with $R$. For $m\in[2, m']$, where $m'=\frac{\epsilon n}{3K}$, we have
\begin{align*}
P_m&\leq \big(\frac{enK}{2}\exp\big(-\frac{(1-\eta)nI}{K}+mI\big)\big)^m\\
&\leq\big(\frac{enK}{2}\exp\big(-\frac{(1-\eta)nI}{K}+mI\big)\big)\big(\frac{enK}{2}\exp\big(-\frac{(1-\eta)nI}{K}+m'I\big)\big)^{m-1}\\
&\leq n\exp\big(-\frac{(1-2\eta)nI}{K}\big)\exp(mI)n^{-\epsilon (m-1)/3}\\
&\leq Rn^{-\epsilon (m-1)/6},
\end{align*}
where we use the fact that $I\lesssim 1$ in the fourth inequality to show $e^{I}n^{-\epsilon/6}<1$ when $n$ is large enough. As a consequence, $\sum_{i=2}^{m'}mP_m=o(R)$, as $\{mP_m\}_{i=2}^{m'}$ is dominated by a fast-decay geometric series.

For $m\in[m',n]$, we have 
\begin{align*}
P_m&\leq\big(\frac{enK}{m'}\exp\big(-\frac{2(1-\eta)nI}{9K}\big)\big)^m\\
&\leq n\exp\big(-\frac{(1-2\eta)nI}{K}\big)\big(\frac{enK}{m'}\exp\big(-\frac{2(1-\eta)nI}{9K}\big)\big)^{m-9}\\
&\leq Rn^{-2(m-9)/9}.
\end{align*}
Since $m'\rightarrow\infty$, $\{mP_m\}_{m\geq m'}$ is dominated by a fast-decay geometric series, which leads to $\sum_{i>m'}^{n}mP_m=o(R)$.
\newline
~\\
(2) If $\limsup_{n\rightarrow\infty}\frac{nI}{K\log n}<1$, there exists a small constant $\epsilon>0$ such that $\frac{(1-\eta)nI}{K\log n}<1-\epsilon$. Let $m_0=n\exp\big(-(1-K^{-\epsilon/2})\frac{(1-\eta)nI}{K}\big)$, which satisfies both $m_0\geq (nK)^{\epsilon/2}$ and $m_0=o(\frac{n}{K^2})$. We are going to show that $\{P_m\}_{m\geq m_0}$ is upper bounded by a fast decaying series $\{Q_m\}_{m\geq m_0}$.

For any $m\in[m_0, m']$, where $m'=\frac{n}{K^{1+\epsilon}}$, we have
\begin{align*}
P_m&\leq \big(\big(\frac{enK}{m_0}\big)\exp\big(-\frac{(1-\eta)nI}{K}+m'I\big)\big)^{m} \\
&\leq\big(\exp\big(\log(nK)+\big((1-K^{-\epsilon/2})-(1-2K^{-\epsilon})\big)\frac{(1-\eta)nI}{K}\big)\big)^m\\
&\leq \exp\big(-\frac{m}{2K^{\epsilon/2}}\frac{(1-\eta)nI}{K}\big),
\end{align*}
which is denoted as $Q_m$. Since $\frac{m_0}{K^{\epsilon/2}}\gg \log n$, we have $\sum_{m=m_0}^{m'}P_m\leq \sum_{m=m_0}^{m'}Q_m\leq m'Q_{m_0}\leq\exp\big(\log n-\frac{m_0}{2K^{\epsilon/2}}\frac{(1-\eta)nI}{K}\big)= o(\frac{m_0}{n})$. For $m'\leq m$, we have
\begin{align*}
P_m \leq \big(\frac{enK}{m'}\exp\big(-\frac{2(1-\eta)nI}{9K}\big)\big)^m\leq  \exp\big(-\frac{nmI}{9K}\big).
\end{align*}
Denote $Q_m=\exp\big(-\frac{nmI}{9K}\big)$, which decays geometrically fast, as $\frac{nI}{K}\rightarrow\infty$. Thus $\sum_{m=m'}^{n}P_m\leq \sum_{m=m'}^{n}Q_m\leq 2Q_{m'}=o(\frac{m_0}{n})$. Consequently,
\begin{align*}
\R(\sigma,\hat\sigma_0)&\leq \frac{m_0}{n}+\p\big(\exists \sigma\in \Theta^0: d(\sigma_0,\sigma)\geq m_0\text{ \& }l(\sigma)\geq l(\sigma_0)\big)\\
&\leq \frac{m_0}{n}+\sum_{m>m_0}^{m'}P_{m'}+\sum_{m>m'}^{n}P_{m}\\
&\leq \frac{m_0}{n} + m'Q_{m_0} + 2Q_{m'}\\
&=\exp\big(-\frac{(1-o(1))nI}{K}\big).
\end{align*}
(3) If $\frac{nI}{K\log n}=1+o(1)$, there exists a positive sequence $w\rightarrow 0$ such that $|\frac{(1-\eta)nI}{K\log n}-1|\ll w$, $\frac{1}{\sqrt{\log n}}\leq w$ and $\frac{wnmI}{K\log K}\rightarrow\infty$. Define $m_0=n\exp\big(-(1-w)\frac{(1-\eta)nI}{K}\big)$. Thus $m_0\geq n^{w/2}\rightarrow\infty$, and $m_0=o(m')$ for $m'=w^2n/K$. We are going to find a fast decay series $\{Q_m\}$ to upper bound $\{P_m\}$. For $m\in[m_0,m'],$
\begin{align*}
P_m&\leq \big(\big(\frac{enK}{m_0}\big)\exp\big(-\frac{(1-\eta)nI}{K}+m'I\big)\big)^{m} \\
&\leq\big(\log(eK)+\frac{(1-w)(1-\eta)nI}{K}-\frac{(1-\eta)nI}{K}+\frac{w^2nI}{K}\big)^m\\
&\leq \exp\big(-\frac{\omega(1-\eta)nmI}{4K}\big),
\end{align*}
which is denoted as $Q_m$. Note that $\omega m_0\geq wn^{w/2}\rightarrow\infty$. We have $Q_{m_0}<1$, and furthermore
\begin{align*}
\sum_{m=m_0}^{m'}P_m\leq\sum_{m=m_0}^{m'}Q_m\leq m'Q_{m_0}\leq \exp\big(\log n-\frac{\omega m_0(1-\eta)nI}{4K}\big)=o(\frac{m_0}{n}).
\end{align*}
For $m\in[m',n]$, we have
\begin{align*}
P_m \leq \big(\frac{enK}{m'}\exp\big(-\frac{2(1-\eta)nI}{9K}\big)\big)^m\leq \exp\big(-\frac{nmI}{9K}\big),
\end{align*}
Let $Q_m=\exp\big(-\frac{nmI}{9K}\big)$, which decays geometrically fast. Then $\sum_{m=m'}^{n}P_m\leq \sum_{m=m'}^{n}Q_m\leq 2Q_{m'}=o(\frac{m_0}{n})$. Hence
\begin{align*}
\R(\sigma,\hat\sigma_0)\leq \frac{m_0}{n}+\sum_{m>m_0}^{m'}P_{m'}+\sum_{m>m'}^{n}P_{m}\leq\exp\big(-\frac{(1-o(1))nI}{K}\big).
\end{align*}

When $K$ is a fixed constant, the proof is nearly identical but with different $m'$ under each scenario. The proof is thus omitted.
\end{proof}

\section{Proofs of Auxiliary Lemmas}\label{sec:proof_lemma}
We prove Lemma \ref{lem:global_local}, Lemma \ref{lem::multiple_test}, Lemma \ref{lem::binomial_tail}, Lemma \ref{lem::number_black_white}, Proposition \ref{prop::chernoff} and Propostion \ref{prop::cardinality} respectively in this section.

\subsection{Proof of Lemma \ref{lem:global_local}}
Before going directly into the proof we define another network operator: (element-wise) permutation. Let $\pi:[1,2,\ldots,n]\rightarrow [1,2,\ldots,n]$ be a permutation. Denote $\Pi$ to be the set consisting of all such permutations, whose cardinality is $n!$. Define $\sigma_{\pi}$ to be a new assignment with
\begin{align*}
\sigma_{\pi}(i)\triangleq\sigma(\pi^{-1}(i)),\forall 1\leq i\leq n.
\end{align*}
It is obvious that for an arbitrary assignment $\sigma\in\Lambda$, each of its permutation $\sigma_\pi$ is also in the parameter space $\Lambda$.

On the other hand, a permutation on the nodes leads to the change of the network. For a network $G$ with an adjacency matrix $A$, define $G_\pi$ as the network after permutation with a new adjacency matrix $A_\pi$, where
\begin{align*}
(A_\pi)_{i,j}=A_{\pi^{-1}(i),\pi^{-1}(j)}.
\end{align*}
Note that $G_\pi$ can be seen as a network sampled from the assignment $\sigma_\pi$, since $(A_\pi)_{i,j}\sim\Ber(\theta_{\pi^{-1}(i),\pi^{-1}(j)})$.

The proof of Lemma \ref{lem:global_local} is mainly by exploring the exchangeability of the network. Any estimator $\hat\sigma$ is a mapping from a network to a length $n$ vector. We use the square brackets $\hat\sigma[G]$ to indicate that the outcome of $\hat\sigma$ is implemented on the network $G$. And $\hat\sigma[G](i)$ is the value of the $i$-th component of $\hat\sigma[G]$, and when the meaning is clear, we write $\hat\sigma(i)$ for simplicity.

Based on $\hat\sigma$, we can always design a new (unless they are the same) procedure by permutation. Given a network $G$, we can either directly apply $\hat\sigma$ (to be more precise, it is $\hat\sigma[G]$), or first permute the network into $G_\pi$, then implement $\hat\sigma$ on it to get $\hat\sigma[G_\pi]$, and then finally ``permute back" to get the estimation in the original order. To be more precise, define procedure $\hat\sigma^\pi$ as
\begin{align*}
\hat\sigma^\pi[G](i)= \hat\sigma[G_\pi](\pi(i)).
\end{align*}
We use the notation $\hat\sigma^\pi(i)$ short for $\hat\sigma^\pi[G](i)$. See Figure \ref{fig::permutation} for the illustration on getting $\hat\sigma^\pi$. 
\begin{figure}[h]
\centering
\includegraphics[trim = 0mm 0mm 0mm 0mm, clip, width=0.95\textwidth]{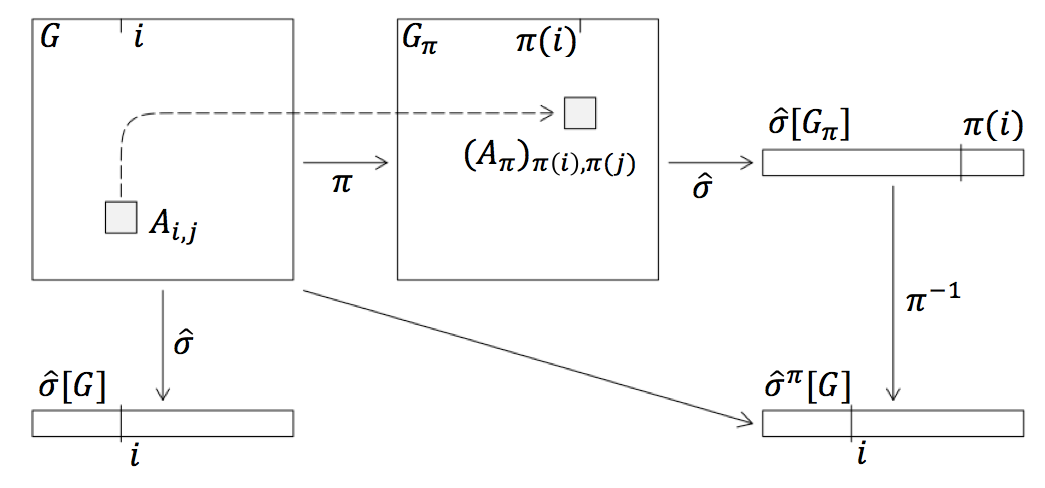}
\caption{Illustration on getting $\hat\sigma^\pi$ based on the original network $G$. All of $\hat\sigma[G]$, $\hat\sigma[G_\pi]$ and $\hat\sigma^\pi[G]$ are demonstrated as $n$-by-1 vectors. It shows $A_{i,j}$ becomes $(A_\pi)_{\pi_i,\pi_j}$ after the permutation $\pi$ of the network. For any specific node $i$ in $G$, its location is changed into $\pi(i)$ in $G_\pi$. The procedure $\hat\sigma[G_\pi]$ estimates the assignemnt of the permuted nodes $\{\pi(i)\}$, while $\hat\sigma^\pi[G]$ estimates the assignment of the original nodes.\label{fig::permutation}}
\end{figure}

Intuitively, due to the exchangeability of $G$, if $\hat\sigma$ is optimal, it should have the same risk as $\hat\sigma^\pi$ for any possible $\pi$. With this trick we are able to show the existence of a universal procedure $\bar\sigma$ which has the equal global risk for all $\sigma\in\Lambda$ and the equal local risk for all $i\in[n]$. Then the proof is completed by the fact that the minimax risk is lower bounded by the Bayes risk.

\begin{proof}[Proof of Lemma \ref{lem:global_local}]
Denote the network to be $G$. Assume $\tilde\sigma$ be one of the estimators that achieve the global Bayes risk, i.e., $B_\tau(\tilde\sigma)=\inf_{\hat\sigma}B_\tau(\hat\sigma)$.
Based on $\tilde\sigma$, we can define a randomized procedure $\bar\sigma$ as $\p(\bar\sigma=\tilde\sigma^\pi)=1/|\Pi|$, for each $\pi\in\Pi$. We will show $\bar\sigma$ is also a global Bayes estimator in terms of $\tau$. For an arbitrary $\sigma\in\Lambda$, we have
\begin{align*}
\E r(\sigma,\bar\sigma)=\frac{1}{n!}\sum_{\pi\in \Pi}\E r(\sigma,\tilde\sigma^\pi).
\end{align*}
Recall that $\R(\sigma,\tilde\sigma^\pi)=\E\inf_{\sigma'\in\Gamma(\tilde\sigma^\pi)} d_H(\sigma,\sigma')$. There exists a one-by-one relation between $\Gamma(\tilde\sigma^\pi)$ and $\Gamma(\tilde\sigma[G_\pi])$, in a the sense that, for any $\sigma'$ from the former set, there is $\sigma''$ in the latter set such that $\sigma''(i)=\sigma'(\pi^{-1}(i)), \forall i\in[n]$, and the reverse also holds. We have the following equation (we add subscript $\sigma$ to explicitly indicate that the expectation is taken with respect to the assignment $\sigma$),
\begin{align*}
\E r(\sigma,\tilde\sigma^\pi)&=\frac{1}{n}\E_\sigma\inf_{\sigma'\in\Gamma(\tilde\sigma^\pi)}\sum_{i=1}^n1\{\sigma(i)\neq\sigma'(i)\}\\
&=\frac{1}{n}\E_\sigma\inf_{\sigma''\in\Gamma(\tilde\sigma[G_\pi])}\sum_{i=1}^n1\{\sigma_\pi(\pi(i))\neq\sigma''(\pi(i))\}\\
&=\frac{1}{n}\E_\sigma\inf_{\sigma''\in\Gamma(\tilde\sigma[G_\pi])}\sum_{i=1}^n1\{\sigma_\pi(i)\neq\sigma''(i)\}.
\end{align*}
The expectation can be further expanded into
\begin{align*}
\E r(\sigma,\tilde\sigma^\pi)=\frac{1}{n}\sum_{G\in \mathbb{G}}\Big(\inf_{\sigma''\in\Gamma(\tilde\sigma[G_\pi])}\sum_{i=1}^n1\{\sigma_\pi(i)\neq\sigma''(i)\}\Big)\p_{\sigma}(G),
\end{align*}
where $\mathbb{G}$ contains all the possible realizations of the graph. Here the subscript of $\p_{\sigma}(G)$ emphasizes that the probability measure is associated with the assignment $\sigma$. Note that $\p_\sigma(G)=\p_{\sigma_\pi}(G_\pi)$ for any $G$ and that the set $\{G_\pi:G\in\mathbb{G}\}$ is exactly equal to $\mathbb{G}$, we have 
\begin{align*}
\E r(\sigma,\tilde\sigma^\pi)&=\frac{1}{n}\sum_{G\in\mathbb{G}}\Big(\inf_{\sigma''\in\Gamma(\tilde\sigma[G_\pi])}\sum_{i=1}^n1\{\sigma_\pi(i)\neq\sigma''(i)\}\Big)\p_{\sigma_\pi}(G_\pi)\\
&=\frac{1}{n}\sum_{G_\pi\in\mathbb{G}}\Big(\inf_{\sigma''\in\Gamma(\tilde\sigma[G_\pi])}\sum_{i=1}^n1\{\sigma_\pi(i)\neq\sigma''(i)\}\Big)\p_{\sigma_\pi}(G_\pi)\\
&=\frac{1}{n}\sum_{G\in\mathbb{G}}\Big(\inf_{\sigma''\in\Gamma(\tilde\sigma[G])}\sum_{i=1}^n1\{\sigma_\pi(i)\neq\sigma''(i)\}\Big)\p_{\sigma_\pi}(G),
\end{align*}
which yields
\begin{align*}
\E r(\sigma,\tilde\sigma^\pi)&=\frac{1}{n}\E_{\sigma_\pi}\inf_{\sigma''\in\Gamma(\tilde\sigma[G])}\sum_{i=1}^n1\{\sigma_\pi(i)\neq\sigma''(i)\}\\
&=\E r(\sigma_\pi,\tilde\sigma).
\end{align*}
Thus
\begin{align*}
B_\tau(\bar\sigma)=\frac{1}{|\Lambda|}\sum_{\sigma\in\Lambda}\Big(\frac{1}{|\Pi|}\sum_{\pi\in \Pi}\E r(\sigma_\pi,\tilde\sigma)\Big)=\frac{1}{|\Pi|}\sum_{\pi\in \Pi}\Big(\frac{1}{|\Lambda|}\sum_{\sigma\in\Lambda}\E r(\sigma_\pi,\tilde\sigma)\Big).
\end{align*}
Since $\{\sigma_\pi:\sigma\in\Lambda\}$ is exactly equal to $\Lambda$ for any $\pi$, we have
\begin{align*}
B_\tau(\bar\sigma)=\frac{1}{|\Pi|}\sum_{\pi\in \Pi}\Big(\frac{1}{|\Lambda|}\sum_{\sigma\in\Lambda}\E r(\sigma,\tilde\sigma)\Big)=\frac{1}{|\Lambda|}\sum_{\sigma\in\Lambda}\Big(\frac{1}{|\Pi|}\sum_{\pi\in \Pi}\E r(\sigma,\tilde\sigma)\Big)=B_\tau(\tilde\sigma).
\end{align*}
Thus $\bar\sigma$ also achieves the minimum Bayes risk. We will show $B_\tau(\bar\sigma(i))= B_\tau(\bar\sigma(j))$ for any $i,j\in[n]$. It is equivalent to define $\bar\sigma$ as
\begin{align*}
\p\Big(\bar\sigma(i)=\tilde\sigma^{\pi}(i)\Big)=\frac{1}{|\Pi|},\forall i\in[n],
\end{align*}
which implies
\begin{align*}
\E r(\sigma(i),\bar\sigma(i))=\frac{1}{|\Pi|}\sum_{\pi\in\Pi}\E r(\sigma(i),\tilde\sigma^{\pi}(i)).
\end{align*}
Note that $\tilde\sigma^\pi(i)=\tilde\sigma[G_{\pi}](\pi(i))$, and $\sigma(i)=\sigma_{\pi}(\pi(i))$. Recall that the definition of the local risk is
\begin{align*}
\E r(\sigma(i),\tilde\sigma^\pi(i))= \E_\sigma\sum_{\sigma'\in S_\sigma(\tilde\sigma^\pi)}\frac{1\{\sigma(i)\neq\sigma'(i)\}}{|S_\sigma(\tilde\sigma^\pi)|}.
\end{align*}
Here recall $S_\sigma(\hat\sigma)\triangleq\{\sigma'\in\Gamma(\hat\sigma):d_H(\sigma,\sigma')=d(\sigma,\hat\sigma)\}$ for any estimator $\hat\sigma$. It is obvious that there exists a one-by-one relation between $S_\sigma(\tilde\sigma^\pi)$ and $S_{\sigma_\pi}(\tilde\sigma[G_\pi])$. For any $\sigma'\in S_\sigma(\tilde\sigma^\pi)$, there is a unique corresponding $\sigma''\in S_{\sigma_\pi}(\tilde\sigma[G_\pi])$ defined as $\sigma''(i)=\sigma'(\pi^{-1}(i)),\forall i\in[n]$, and the reverse also holds. Thus the event $\{\sigma(i)\neq\sigma'(i)\}$ is equivalent to $\{\sigma_\pi(\pi(i))\neq\sigma''(\pi(i))\}$, and $|S_\sigma(\tilde\sigma^\pi)|=|S_{\sigma_\pi}(\tilde\sigma[G_\pi])|$. We have
\begin{align*}
\E r(\sigma(i),\tilde\sigma^\pi(i))&= \E_\sigma\sum_{\sigma''\in S_{\sigma_\pi}(\tilde\sigma[G_\pi])}\frac{1\{\sigma_\pi(\pi(i))\neq\sigma''(\pi(i))\}}{|S_{\sigma_\pi}(\tilde\sigma[G_\pi])|}.
\end{align*} 
By the same argument as the previous one, together with the fact that $\p_{\sigma}(G)=\p_{\sigma_\pi}(G_\pi)$, we expand the expectation and then have
\begin{align*}
\E r(\sigma(i),\tilde\sigma^\pi(i))&=\sum_{G\in\mathbb{G}}\Big(\sum_{\sigma''\in S_{\sigma_\pi}(\tilde\sigma[G_\pi])}\frac{1\{\sigma_\pi(\pi(i))\neq\sigma''(\pi(i))\}}{|S_{\sigma_\pi}(\tilde\sigma[G_\pi])|}\Big)\p_\sigma(G)\\
&=\sum_{G\in\mathbb{G}}\Big(\sum_{\sigma''\in S_{\sigma_\pi}(\tilde\sigma[G_\pi])}\frac{1\{\sigma_\pi(\pi(i))\neq\sigma''(\pi(i))\}}{|S_{\sigma_\pi}(\tilde\sigma[G_\pi])|}\Big)\p_{\sigma_\pi}(G_\pi)\\
&=\sum_{G\in\mathbb{G}}\Big(\sum_{\sigma''\in S_{\sigma_\pi}(\tilde\sigma[G])}\frac{1\{\sigma_\pi(\pi(i))\neq\sigma''(\pi(i))\}}{|S_{\sigma_\pi}(\tilde\sigma[G])|}\Big)\p_{\sigma_\pi}(G).\\
\end{align*}
Thus
\begin{align*}
\E r(\sigma(i),\tilde\sigma^\pi(i))&= \E_{\sigma_\pi}\sum_{\sigma''\in S_{\sigma_\pi}(\tilde\sigma[G_\pi])}\frac{1\{\sigma_\pi(\pi(i))\neq\sigma''(\pi(i))\}}{|S_{\sigma_\pi}(\tilde\sigma[G])|}\\
&=\E r(\sigma_\pi(\pi(i)),\tilde\sigma(\pi(i))).
\end{align*}
This gives
\begin{align*}
\E r(\sigma(i),\bar\sigma(i))=\frac{1}{|\Pi|}\sum_{\pi\in \Pi}\E r\big(\sigma_\pi(\pi(i)),\tilde\sigma(\pi(i))\big),\forall i\in[n].
\end{align*}
Then for the local risk we have
\begin{align*}
B_\tau(\bar\sigma(i))&=\frac{1}{|\Lambda|}\sum_{\sigma\in\Lambda}\Big(\frac{1}{|\Pi|}\sum_{\pi\in \Pi}\E r\big(\sigma_\pi(\pi(i)),\tilde\sigma(\pi(i))\big)\Big)\\
&=\frac{1}{|\Pi|}\sum_{\pi\in \Pi}\Big(\frac{1}{|\Lambda|}\sum_{\sigma\in\Lambda}\E r\big(\sigma_\pi(\pi(i)),\tilde\sigma(\pi(i))\big)\Big)\\
&=\frac{1}{|\Pi|}\sum_{\pi\in \Pi}\Big(\frac{1}{|\Lambda|}\sum_{\sigma\in\Lambda}\E r\big(\sigma(\pi(i)),\tilde\sigma(\pi(i))\big)\Big)\\
&=\frac{1}{|\Lambda|}\sum_{\sigma\in\Lambda}\Big(\frac{1}{|\Pi|}\sum_{\pi\in \Pi}\E r\big(\sigma(\pi(i)),\tilde\sigma(\pi(i))\big)\Big)\\
&=\frac{1}{|\Lambda|}\sum_{\sigma\in\Lambda}\Big(\frac{1}{n}\sum_{l=1}^n\E r\big(\sigma(l),\tilde\sigma(l)\big)\Big),
\end{align*}
where in the third equation we again use the fact that $\{\sigma_\pi:\sigma\in\Lambda\}$ is exactly equal to $\Lambda$ for any $\pi$. So we conclude $B_\tau(\bar\sigma(i))= B_\tau(\bar\sigma(j))$ for any $i,j\in[n]$. Due to the equality
\begin{align*}
\E r(\sigma,\hat\sigma)&=\E \inf_{\delta}\sum_{i=1}^n\frac{1\{(\delta\circ\hat\sigma)(i)\neq \sigma(i)\}}{n}\\
&=\E \frac{1}{|S_\sigma(\hat\sigma)|}\sum_{\sigma'\in S_\sigma(\hat\sigma)}\sum_{i=1}^n\frac{1\{i:\sigma'(i)\neq\sigma(i)\}}{n}\\
&=\frac{1}{n}\sum_{i=1}^n\E\sum_{\sigma'\in S_\sigma(\hat\sigma)}\frac{1\{i:\sigma'(i)\neq\sigma(i)\}}{|S_\sigma(\hat\sigma)|}\\
&=\frac{1}{n}\sum_{i=1}^n\E r(\sigma(i),\hat\sigma(i)),
\end{align*}
we have $B_\tau(\bar\sigma)=\sum_{i=1}^n B_\tau(\bar\sigma(i))/n$, which leads to $\inf_{\hat\sigma}B_\tau(\hat\sigma)=B_\tau(\bar\sigma)=B_\tau(\bar\sigma(1))\geq\inf_{\hat\sigma}B_\tau(\hat\sigma(1))$. We omit the proof of the other direction of the equality stated in the lemma, which uses  a nearly identical argument. The proof is complete.
\end{proof}

\subsection{Proof of Lemma \ref{lem::multiple_test}}
First consider the case with $K\geq 3$. Define $\Theta^L_1=\{(\sigma,\{\theta_{i,j}\})\in\Theta^L: n_{\sigma(1)}=\lfloor \frac{n}{K}\rfloor + 1\}$. So for each $\sigma\in\Theta^L_1$, the community containing the first node always has size $\lfloor \frac{n}{K}\rfloor + 1$. We will show the ratio of the cardinality of $\Theta^L_1$ against that of $\Theta^L$ is a constant. Denote $x_1=\lfloor n/K\rfloor K_1$ and $x_2=(\lfloor n/K\rfloor+1)K_2$, then
\begin{align*}
|\Theta^L|=C\binom{n}{x_2}\binom{n-x_2}{x_1} \text{ and } |\Theta^L_1|=C\binom{n-1}{x_2-1}\binom{n-x_2}{x_1},
\end{align*}
where $C$ is the number of combinations to select $x_1$ balls into $K_1$ bins with size $\lfloor \frac{n}{K}\rfloor$, $x_2$ balls into $K_2$ bins with size $\lfloor \frac{n}{K}\rfloor+1$, and another $n-x_1-x_2$ balls into $K_3$ bins with size $\lfloor \frac{n}{K}\rfloor-1$. Thus
\begin{align*}
\frac{|\Theta^L_1|}{|\Theta^L|}=\frac{\binom{n-1}{x_2-1}}{\binom{n}{x_2}}=\frac{x_2}{n}\geq\epsilon.
\end{align*}
It is equivalent to the probability that the first node is assigned to the $K_2$ bins with size $\lfloor \frac{n}{K}\rfloor+1$. Then
\begin{align*}
B_\tau(\hat\sigma(1))&\geq \frac{1}{|\Theta^L|}\sum_{\sigma\in\Theta^L_1}\E r(\sigma(1),\hat\sigma(1))\geq \frac{\epsilon}{|\Theta^L_1|}\sum_{\sigma\in\Theta^L_1}\E r(\sigma(1),\hat\sigma(1)).
\end{align*}
For each $\sigma_0\in\Theta^L_1$, let $\kappa_0(\sigma_0)=\sigma_0(1)$ be the index of community that the first node belongs to. And let $\kappa(\sigma_0)$ be the indices of communities whose sizes are $\lfloor \frac{n}{K}\rfloor$, i.e.,
\begin{align*}
\kappa(\sigma_0)=\{k\in[K]:n_k=\lfloor \frac{n}{K}\rfloor\}.
\end{align*}
Note that $\kappa_0(\sigma_0)\notin \kappa(\sigma_0)$. If we replace $\sigma_0(1)$ by any $k\in\kappa(\sigma_0)$ while keep the value of the rest of nodes, then we get a new assignment also contained in $\Theta_1^L$ and has distance 1 from $\sigma_0$. In particular, we use the following procedure to get a new assignment $\sigma[\sigma_0]$ based on $\sigma_0$:
\begin{align*}
\sigma[\sigma_0](1)=\begin{cases}
\min\{k\in\kappa(\sigma_0):k>\kappa_0(\sigma_0)\}\text{ if }\max \kappa(\sigma_0)>\kappa_0(\sigma_0);\\
\min\kappa(\sigma_0)\text{ if }\max \kappa(\sigma_0)<\kappa_0(\sigma_0),
\end{cases}
\end{align*}
and $\sigma[\sigma_0](i)=\sigma_0(i)$ for all $i\geq 2$. It is clear that $\sigma[\sigma_0]\in\Theta_1^L$ and $d_H(\sigma_0,\sigma[\sigma_0])=1$. It is also guaranteed that for any $\sigma_0,\sigma_1\in\Theta_1^L$ and $\sigma_0\neq \sigma_1$, the new assignments are also different $\sigma[\sigma_0]\neq\sigma[\sigma_1]$. This leads to that $\Theta^L$ is equal to the set $\{\sigma[\sigma_0]:\sigma_0\in\Theta^L\}$, and hence
\begin{align*}
B_\tau(\hat\sigma(1))&\geq \frac{\epsilon}{2|\Theta^L_1|}\sum_{\sigma_0\in\Theta^L_1}2\E r(\sigma_0(1),\hat\sigma(1))\\
&\geq \frac{\epsilon}{|\Theta^L_1|}\sum_{\sigma_0\in\Theta^L_1}\frac{1}{2}\big(\E r(\sigma_0(1),\hat\sigma(1))+\E r(\sigma[\sigma_0](1),\hat\sigma(1))\big).
\end{align*}
We are going to derive the Bayes risk $\inf_{\hat\sigma}\frac{1}{2}\big(\E r(\sigma_0(1),\hat\sigma(1))+\E r(\sigma[\sigma_0](1),\hat\sigma(1))\big)$ for a given $\sigma_0\in\Theta^L$. Let $\tilde\sigma$ be any estimator achieving the infimum. Since $d_H(\sigma_0,\sigma[\sigma_0])=1$, we have $r(\sigma_0(1),\tilde\sigma(1))=d_H(\sigma_0(1),\tilde\sigma(1))$ and a similar equation holds for $\sigma[\sigma_0]$. The estimator $\tilde\sigma(1)$ can be interpreted as the Bayes estimator with respect to the zero-one loss. Then $\tilde\sigma(1)$ must be the mode of the posterior distribution. Let $J_0$ to be the set $\{u\in[n]\setminus\{1\}:\sigma_0(u)=\sigma_0(1)\}$, and $J_1=\{u\in[n]:\sigma_0(u)=\sigma[\sigma_0](1)\}$. For a given adjacency matrix $A$, the conditional distributions are
\begin{align*}
\p(A|\sigma_0)=\prod_{u\in J_0}(\frac{a}{n})^{A_{1,u}}(1-\frac{a}{n})^{1-A_{1,u}}\prod_{u\in J_1}(\frac{b}{n})^{A_{1,u}}(1-\frac{b}{n})^{1-A_{1,u}}f(A^{C}),
\end{align*}
and
\begin{align*}
\p(A|\sigma[\sigma_0])=\prod_{u\in J_1}(\frac{a}{n})^{A_{1,u}}(1-\frac{a}{n})^{1-A_{1,u}}\prod_{u\in J_0}(\frac{b}{n})^{A_{1,u}}(1-\frac{b}{n})^{1-A_{1,u}}f(A^{C}).
\end{align*}
Here $A^C$ consists all the rest of the entries: $A^C=\{(u,v):v>u\geq 2\text{, or } u=1\text{ and }v\notin J_0\cup J_1\}$. It is obvious that $f(A^C)$ is invariant to the choice of $\sigma_0$ or $\sigma[\sigma_0]$. Thus 
\begin{align*}
\tilde\sigma(1)=\begin{cases}
\sigma_0(1),\text{ if }\sum_{u\in J_0}A_{1,u}\geq \sum_{u\in J_1}A_{1,u},\\
\sigma[\sigma_0](1),\text{ if }\sum_{u\in J_0}A_{1,u}< \sum_{u\in J_1}A_{1,u}.
\end{cases}
\end{align*}
Thus $\E r(\sigma_0(1),\hat\sigma(1))=\p_{\sigma_0}\big(\sum_{u\in J_0}A_{1,u}< \sum_{u\in J_1}A_{1,u}\big)\geq\p\big(\sum_{u=1}^{\lfloor n/K\rfloor}X_u\geq \sum_{u=1}^{\lfloor n/K\rfloor}Y_u\big)$, and $\E r(\sigma[\sigma_0](1),\hat\sigma(1))\big)=\p_{\sigma[\sigma_0]}\big(\sum_{u\in J_0}A_{1,u}\geq \sum_{u\in J_1}A_{1,u}\big)\geq\p\big(\sum_{u=1}^{\lfloor n/K\rfloor}X_u\geq \sum_{u=1}^{\lfloor n/K\rfloor}Y_u\big)$. Consequently,
\begin{align*}
\frac{1}{2}\big(\E r(\sigma_0(1),\hat\sigma(1))+\E r(\sigma[\sigma_0](1),\hat\sigma(1))\big)&\geq\p\Big(\sum_{u=1}^{\lfloor n/K\rfloor}X_u\geq \sum_{u=1}^{\lfloor n/K\rfloor}Y_u\Big).
\end{align*}
The above inequality holds for each $\sigma_0\in\Theta^L$. Hence
\begin{align*}
\inf_{\hat\sigma}B_\tau(\hat\sigma(1))&\geq\frac{\epsilon}{|\Theta^L_1|}\sum_{\sigma_0\in\Theta^L_1}\inf_{\hat\sigma}\frac{1}{2}\big(\E r(\sigma_0(1),\hat\sigma(1))+\E r(\sigma[\sigma_0](1),\hat\sigma(1))\big)\\
&\geq \frac{\epsilon}{|\Theta^L_1|}\sum_{\sigma_0\in\Theta^L_1}\frac{1}{2}\big(\E r(\sigma_0(1),\tilde\sigma(1))+\E r(\sigma[\sigma_0](1),\tilde\sigma(1))\big)\\
&\geq \epsilon \p\Big(\sum_{u=1}^{\lfloor n/K\rfloor}X_u\geq \sum_{u=1}^{\lfloor n/K\rfloor}Y_u\Big).
\end{align*}

For the case $K=2$, we re-define $\Theta^L_1$ and show that its cardinality is same with that of $\Theta^L$ up to a constant factor. \textit{(1)} If $\frac{n}{2}\neq\lfloor\frac{n}{2}\rfloor$, then define $\Theta^L_1=\{(\sigma,\{\Theta^L_{i,j}\})\in\Theta^L:n_{\sigma(1)}=\lceil\frac{n}{2}\rceil\}$. Then $|\Theta^L_1|/|\Theta^L|=1/2$.
\textit{(2)} If $\frac{n}{2}=\lfloor\frac{n}{2}\rfloor$, then define $\Theta^L_1=\{(\sigma,\{\Theta^L_{i,j}\})\in\Theta^L:n_{\sigma(1)}>\frac{n}{2}\}$. Then
\begin{align*}
\frac{|\Theta^L_1|}{\Theta^L}=1-\frac{|\Theta^L\setminus\Theta^L_1|}{|\Theta^L|}=1-\frac{\binom{n-1}{n/2-1}}{\binom{n}{n/2}+2\binom{n}{n/2+1}}=1-\frac{(n/2-1)/n}{1+\frac{n/2}{n/2+1}}>\frac{1}{2}.
\end{align*}
Then with exactly the same argument used for $K\geq 3$ we finish the proof.

\subsection{Proof of Lemma \ref{lem::binomial_tail}}
\textit{(1)} First consider the case when $\frac{nI}{K\log K}\rightarrow\infty$. Let $p(x)$ be the probability mass function of $Z_{i}$, and $M(t)$ be the moment generating function of $Z_{i}$. That is
\begin{align*}
M(t)&=\E e^{tX_i}\E e^{-tY_i}\\
&=\big(e^t\frac{b}{n}+1-\frac{b}{n}\big)\big(e^{-t}\frac{a}{n}+1-\frac{a}{n}\big).
\end{align*}
The minimum of $M(t)$ is achieved at $t^\star=\frac{1}{2}\log\frac{a(1-b/n)}{b(1-a/n)}$, with $M(t^\star)=\big(\sqrt{\frac{a}{n}\frac{b}{n}}+\sqrt{(1-\frac{a}{n})(1-\frac{b}{n})}\big)^2$. This gives $I=-\log M(t^\star)=\max_t(-\log M(t))$. Let $\delta$ be a positive number which may depend on $n$. Denote $S_{n'}=\sum_{i=1}^{n'}Z_{i}$. Then
\begin{align*}
\p(S_{n'}\geq 0)&\geq \sum_{n'\delta>S_{n'}\geq 0}\prod_{i=1}^{n'}p(z_i) \\
&\geq \frac{M^{n'}( t^\star )}{\exp(n' t^\star \delta)}\sum_{n'\delta>S_{n'}\geq 0}\prod_{i=1}^{n'}\frac{\exp( t^\star z_i)p(z_i)}{M( t^\star )},
\end{align*}
where we use the fact that $\exp(n' t^\star \delta)\geq \exp(t^\star \sum_i z_i) \geq \prod \exp(t^\star z_i)$ when $\sum_i z_i< n'\delta$. Denote $q(w)=\frac{\exp( t^\star w)p(w)}{M( t^\star )}$. Then
\begin{align*}
\p(S_{n'}\geq 0)&\geq \frac{M^{n'}( t^\star )}{\exp(n' t^\star \delta)}\sum_{n'\delta>S_{n'}\geq 0}\prod_{i=1}^{n'}q(z_i) \\
&=\exp\left(-n'I\right)\exp(-n' t^\star \delta)\sum_{n'\delta>S_{n'}\geq 0}\prod_{i=1}^{n'}q(z_i).
\end{align*}
Note that $q(w)$ is a probability mass function, as $\sum_w \frac{\exp( t^\star w)p(w)}{M( t^\star )}=1$. Let $W_{1},W_{2},\ldots,W_{n'}$ be i.i.d random variable with probability mass function $q(w)$, then
\begin{align*}
\p(S_{n'}\geq 0)\geq \exp\left(-n'I\right)\exp(-n' t^\star \delta)\p\Big(\delta>\frac{1}{n'}\sum_{i=1}^{n'}W_{i}\geq 0\Big).
\end{align*}
A closer look on $W_{1}$ gives
\begin{align*}
\p(W_{1}=1)=\p(W_{1}=-1)=\frac{1}{M( t^\star )}\sqrt{\frac{a}{n}\frac{b}{n}(1-\frac{a}{n})(1-\frac{b}{n})},
\end{align*}
and $\p(W_{1}=0)=1-\p(W_{1}=1)-\p(W_{1}=-1)$. Thus $\E W_{1}=0$ and
\begin{align*}
\text{Var}(W_{1})=\frac{2}{M( t^\star )}\sqrt{\frac{a}{n}\frac{b}{n}(1-\frac{a}{n})(1-\frac{b}{n})}.
\end{align*}
Denote $V=\text{Var}(\sum_{i=1}^{n'}W_{i}/n')=\text{Var}(W_{1})/n'$. We will later show that $I/(t^\star \sqrt{V})\rightarrow\infty$. Now if it holds, then define $\delta=V^{\frac{1}{4}}I^{\frac{1}{2}}( t^\star )^{-\frac{1}{2}}$. It satisfies $\sqrt{V}=o(\delta)$. Chebyshev's inequality yields
\begin{align*}
\p\Big(\Big|\frac{1}{n'}\sum_{i=1}^{n'}W_{i}\Big|\geq \delta\Big)\leq \frac{V}{\delta^2}=o(1).
\end{align*}
By the fact that the distribution of $\frac{1}{n'}\sum_{i=1}^{n'}W_{i}$ is symmetric, we have
\begin{align*}
\p\Big(\delta>\frac{1}{n'}\sum_{i=1}^{n'}W_{i}\geq 0\Big)=\frac{1}{2}\Big(1-\p\Big(\Big|\frac{1}{n'}\sum_{i=1}^{n'}W_{i}\Big|\geq \delta\Big)\Big)\rightarrow \frac{1}{2}.
\end{align*}

To prove $I/(t^\star \sqrt{V})\rightarrow\infty$, first consider the case with $a\asymp b$. Since $I\asymp(a-b)^2/na$, and $t^\star=\frac{1}{2}\log((1+\frac{a-b}{b})(1+\frac{a-b}{n(1-a/n)}))\asymp\frac{a-b}{a}$, and $\sqrt{V}\asymp\sqrt{aK}/n$, we have $\frac{I}{t^\star \sqrt{V}}\asymp\frac{a-b}{\sqrt{aK}}\rightarrow\infty$, implied by the fact $\frac{nI}{K}\asymp\frac{a-b}{\sqrt{aK}}\rightarrow\infty$. On the other hand, if $a/b\rightarrow\infty$ (recall we assume $b>\epsilon>0$ and $a/n<1-\epsilon$), we have $I\asymp a/n$ and $M(t^\star)\asymp 1$. Note that $(\log\frac{a}{b})(\frac{b}{a})^{\frac{1}{4}}$ goes to 0, hence $t^\star \sqrt{V}=o(\sqrt{aK}/n)$. Then $\frac{I}{t^\star \sqrt{V}}\gg \sqrt{a/K}$. Since $nI/K\asymp a/K\rightarrow\infty$, $\frac{I}{t^\star \sqrt{V}}$ also goes to infinity.

\textit{(2)} If $\frac{nI}{K}=O(1)$, we can choose $\delta$ such that $nt^\star\delta/K$ is also a constant. Then by considering the case $a\asymp b$ and $a/b\rightarrow\infty$ separately, we have $\frac{\delta}{\sqrt{V}}\asymp\frac{K}{nt^\star\sqrt{V}}\asymp \frac{K}{nI}$ with a similar argument used above. Thus $\p(S_{n'}>0)$ is a constant.

\subsection{Proof of Lemma \ref{lem::number_black_white}}
Due to the symmetry between $\sigma$ and $\sigma_0$ (both are in the same parameter space), we have $\alpha(\sigma;\sigma_0)=\gamma(\sigma_0;\sigma)$ and $\gamma(\sigma;\sigma_0)=\alpha(\sigma_0;\sigma)$. It is sufficient to get the desired lower bound for $\gamma(\sigma;\sigma_0)$, as the same bound automatically holds for $\alpha(\sigma;\sigma_0)$.

By the definition of $\Theta^0$ there must exist a $\eta_1\rightarrow 0$ such that $|\frac{n_k}{n/K}-1|\leq \eta_1$ for every $k\in[K]$. First consider $m\leq\frac{n}{2K}$. Without loss of generality, let $\sigma$ satisfy
\begin{align*}
\sigma(i)=k,\forall i\in[\sum_{j=1}^{k-1}n'_{j}+1,\sum_{j=1}^{k}n'_{j}],
\end{align*}
where $\{n'_{k}\}_{k=1}^K$ are the sizes of communities in $\sigma$. Recall $\{n_k\}_{k=1}^K$ are the true community sizes in $\sigma_0$. Define $m_k=|\{i:\sigma(i)=k,\sigma_0(i)\neq k\}|$, then $m=\sum_k m_k$. For $k\in[K]$, define
\begin{align*}
\gamma_{k}(\sigma;\sigma_{0})&=|\{(i,j):\sigma(i)=\sigma(j)=k,\sigma_0(i)\neq \sigma_0(j),i<j\}| \\
&=\Big|\Big\{(i,j):\sigma_0(i)\neq \sigma_0(j),\sum_{j=1}^{k-1}n'_{j}+1\leq i<j\leq \sum_{j=1}^{k}n'_{j}\Big\}\Big|.
\end{align*}
Obviously $\gamma(\sigma;\sigma_{0})=\sum_{k=1}^{K}\gamma_{k}(\sigma;\sigma_{0})$. We have $\gamma_k(\sigma;\sigma_0)\geq |\{i:\sigma(i)=k, \sigma_0(i)=k\}||\{i:\sigma(i)=k, \sigma_0(i)\neq k\}|= (n_k-m_k)m_k$. Then
\begin{align*}
\gamma(\sigma;\sigma_{0})\geq \sum_k m_k(n_k-m_k)\geq\frac{(1-o(1))mn}{K}-\sum_k m_k^2\geq \frac{(1-o(1))mn}{K}-m^2.
\end{align*}
\\
Now consider the case $m>\frac{n}{2K}$. Define $m_{k,k'}=|\{i:\sigma(i)=k,\sigma_0(i)= k'\}|$ for any $k,k'\in[K]$. It is obvious that equations $m_k=\sum_{k'\neq k}m_{k,k'}$, $n'_k=m_k+m_{k,k}$ and $n_{k'}=\sum_{k}m_{k,k'}$ hold for any $k$ and $k'$. 

It can be shown that we cannot find an pair of $(k,k')$ such as $k\neq k'$ and $m_{k,k'}> \frac{2(1+\eta_1)n}{3K}$. Otherwise, if $m_{k,k'}> \frac{2(1+\eta_1)n}{3K}$, then $m_{k',k'}\leq n_{k'}-m_{k,k'}<\frac{(1+\eta_1)n}{3K}$. Then we can exchange the label of $k$ and $k'$ to get a new estimation $\sigma'$. Compared with $\sigma$, this helps correctly recover at least $m_{k,k'}-(n'_k-m_{k,k'})-m_{k',k'}>0$ nodes. Since $\sigma'\in\Gamma(\sigma)$, then $m=d(\sigma_0,\sigma)\leq d_H(\sigma_0,\sigma')<m$, which leads to a contradiction.

So we have $m_{k,k'}\leq \frac{2(1+\eta_1)n}{3K}$ for all $k\neq k'$. For a given $m_k$, we have
\begin{align*}
\frac{\gamma_{k}(\sigma;\sigma_{0})}{n'_k m_k}=\frac{\frac{1}{2}(n_k^{'2}-\sum_{k'}m_{k,k'}^2)}{n'_k m_k},
\end{align*}
with a constrain $m_k=\sum_{k'\neq k}m_{k,k'}$. When $m_k\leq \frac{2(1+\eta_1)n}{3K}$, it can be shown that
\begin{align*}
\frac{\gamma_{k}(\sigma;\sigma_{0})}{n'_k m_k}\geq \frac{\frac{1}{2}(n_k^{'2}-(n'_k-m_k)^2-m_k^2)}{n'_k m_k}=\frac{n'_k-m_k}{n'_k}\geq \frac{(1-5\eta_1)\frac{n}{K}}{3n'_k}.
\end{align*}
And when $m_k\geq \frac{2(1+\eta_1)n}{3K}$,
\begin{align*}
\frac{\gamma_{k}(\sigma;\sigma_{0})}{n'_k m_k}&\geq \frac{\frac{1}{2}(n_k^{'2}-(n'_k-m_k)^2-(m_k-\frac{2(1+\eta_1)n}{3K})^2-(\frac{2(1+\eta_1)n}{3K})^2)}{n'_k m_k}\\
&\geq \frac{m_k(n'_k-m_k)+\frac{2(1+\eta_1)n}{3K}(m_k-\frac{2(1+\eta_1)n}{3K})}{n'_k m_k}\\
&\geq \frac{2(1-5\eta_1)\frac{n}{K}}{9n'_k}.
\end{align*}
Then sum up over all $k$ and we get $\gamma(\sigma;\sigma_0)\geq \frac{2(1-5\eta_1)nm}{9K}$. By choosing $\eta=5\eta_1$ the proof is complete.

\subsection{Proofs of Proposition \ref{prop::chernoff} and \ref{prop::cardinality}}
We first present proposition \ref{prop::homo_inhomo}, which is easy to be verified by coupling. It is helpful for the proof of Proposition \ref{prop::chernoff}.
\begin{proposition}\label{prop::homo_inhomo}
Let $\alpha$ and $\gamma$ be arbitrary positive integers, and $m$ take any value in $\mathbb{R}$. Define series of independent variables $\{X_i\}_{i=1}^\alpha$, $\{Y_i\}_{i=1}^\gamma$, $\{U_i\}_{i=1}^\alpha$ and $\{V_i\}_{i=1}^\gamma$. Let $U_i\sim\Ber(p_i)$, $V_i\sim\Ber(q_i)$, $X_i\sim\Ber(p)$ and $Y_i\sim\Ber(q)$ with $\min p_i\geq p$ and $\max q_i\leq q$. Then
\begin{align*}
\p\Big(m+\sum_{i=1}^\alpha U_i\leq \sum_{i=1}^\gamma V_i\Big)\leq \p\Big(m+\sum_{i=1}^\alpha X_i\leq \sum_{i=1}^\gamma Y_i\Big).
\end{align*}
\end{proposition}

\begin{proof}[Proof of Proposition \ref{prop::chernoff}]
Let $\{X_i\}_{1\leq i\leq \gamma}$ be i.i.d $\Ber( \frac{b}{n} )$ random variables and $\{Y_i\}_{1\leq i\leq \alpha}$ be i.i.d $\Ber( \frac{a}{n} )$ random variables, and $\{X_i\}\perp\{Y_i\}$. Then by Proposition \ref{prop::homo_inhomo}, we have
\begin{align*}
\p(T(\sigma)\geq T(\sigma_0))\leq \p\Big(\sum_{i=1}^{\gamma}X_i-\sum_{i=1}^{\alpha}Y_i\geq \lambda(\gamma-\alpha)\Big).
\end{align*}
As an application of Markov inequality,
\begin{align*}
\p(T(\sigma)\geq T(\sigma_0))&\leq \p\Big(\exp\Big(t\sum^{\gamma}X_i-t\sum^{\alpha}Y_i\Big)\geq \exp(t\lambda(\gamma-\alpha)\Big)\Big)\\
&\leq e^{-t\lambda(\gamma-\alpha)}\Big(\E e^{tX_1}\Big)^\gamma\Big(\E e^{-tY_1}\Big)^\alpha\\
&=\Big(\E e^{tX_1}\E e^{-tY_1}\Big)^{(1-w)\alpha+w\gamma} \Big(\frac{(\E e^{tX_1})^{1-w}}{(\E e^{-tY_1})^{w}}e^{-t\lambda}\Big)^{\gamma-\alpha}
\end{align*}
holds for any $t>0$. Choose $t=t^\star$. Then $\E e^{t^\star X_1}\E e^{-t^\star Y_1}=e^{-I}$, and $\frac{(\E e^{t^\star X_1})^{1-w}}{(\E e^{-t^\star Y_1})^{w}}e^{-t^\star\lambda}$ is exactly equal to 1. Thus $\p(T(\sigma)\geq T(\sigma_0))\leq e^{-\min\{\alpha, \gamma\} I}$.
\end{proof}

\begin{proof}[Proof of Proposition \ref{prop::cardinality}]
Without loss of generality we assume that $d_H(\sigma,\sigma_{0})=d(\sigma,\sigma_{0})$. Then $\sigma$ assigns $m$ nodes with different values from $\sigma_{0}$, and there are $K$ possible values for each node. Thus
\begin{align*}
\Big|\Big\{\Gamma:\exists \sigma\in\Gamma\text{ s.t. }d(\sigma,\sigma_{0})=m\Big\}\Big|\leq\binom{n}{m} K^m\leq\left(\frac{enK}{m}\right)^{m}.
\end{align*}
In addition, since each node has at most $K$ possible choices, we have a naive bound for the cardinality of ${\Gamma}$ as $|\{\Gamma\}|\leq K^{n}$.
\end{proof}

\begin{supplement}
\sname{Supplement A}\label{suppA}
\stitle{Supplement to ``Mimimax Rates of Community Detection in Stochastic Block Models''}
\slink[url]{url to be specified}
\sdescription{In the supplement \cite{supplement}, we provide proofs  for Theorems \ref{thm:lower_beta} and \ref{thm:upper_beta}, which extend the minimax results of Theorems \ref{thm:lower} and \ref{thm:upper} to a larger parameter space $\Theta$.}
\end{supplement}

\bibliographystyle{plainnat}
\bibliography{minimax}

\begin{thebibliography}{31}
\providecommand{\natexlab}[1]{#1}
\providecommand{\url}[1]{\texttt{#1}}
\expandafter\ifx\csname urlstyle\endcsname\relax
  \providecommand{\doi}[1]{doi: #1}\else
  \providecommand{\doi}{doi: \begingroup \urlstyle{rm}\Url}\fi

\bibitem[Abbe and Sandon(2015)]{abbe2015community}
Emmanuel Abbe and Colin Sandon.
\newblock Community detection in general stochastic block models: fundamental
  limits and efficient recovery algorithms.
\newblock \emph{arXiv preprint arXiv:1503.00609}, 2015.

\bibitem[Albert et~al.(1999)Albert, Jeong, and
  Barab{\'a}si]{albert1999internet}
R{\'e}ka Albert, Hawoong Jeong, and Albert-L{\'a}szl{\'o} Barab{\'a}si.
\newblock Internet: Diameter of the world-wide web.
\newblock \emph{Nature}, 401\penalty0 (6749):\penalty0 130--131, 1999.

\bibitem[Amini et~al.(2013)Amini, Chen, Bickel, Levina,
  et~al.]{amini2013pseudo}
Arash~A Amini, Aiyou Chen, Peter~J Bickel, Elizaveta Levina, et~al.
\newblock Pseudo-likelihood methods for community detection in large sparse
  networks.
\newblock \emph{The Annals of Statistics}, 41\penalty0 (4):\penalty0
  2097--2122, 2013.

\bibitem[Barabasi and Oltvai(2004)]{barabasi2004network}
Albert-Laszlo Barabasi and Zoltan~N Oltvai.
\newblock Network biology: understanding the cell's functional organization.
\newblock \emph{Nature Reviews Genetics}, 5\penalty0 (2):\penalty0 101--113,
  2004.

\bibitem[Bickel et~al.(2013)Bickel, Choi, Chang, Zhang,
  et~al.]{bickel2013asymptotic}
Peter Bickel, David Choi, Xiangyu Chang, Hai Zhang, et~al.
\newblock Asymptotic normality of maximum likelihood and its variational
  approximation for stochastic blockmodels.
\newblock \emph{The Annals of Statistics}, 41\penalty0 (4):\penalty0
  1922--1943, 2013.

\bibitem[Bickel and Chen(2009)]{bickel2009nonparametric}
Peter~J Bickel and Aiyou Chen.
\newblock A nonparametric view of network models and newman--girvan and other
  modularities.
\newblock \emph{Proceedings of the National Academy of Sciences}, 106\penalty0
  (50):\penalty0 21068--21073, 2009.

\bibitem[Cai and Li(2014)]{cai2014robust}
Tony Cai and Xiaodong Li.
\newblock Robust and computationally feasible community detection in the
  presence of arbitrary outlier nodes.
\newblock \emph{arXiv preprint arXiv:1404.6000}, 2014.

\bibitem[Chen and Xu(2014)]{chen2014statistical}
Yudong Chen and Jiaming Xu.
\newblock Statistical-computational tradeoffs in planted problems and submatrix
  localization with a growing number of clusters and submatrices.
\newblock \emph{arXiv preprint arXiv:1402.1267}, 2014.

\bibitem[Chin et~al.(2015)Chin, Rao, and Vu]{chin2015stochastic}
Peter Chin, Anup Rao, and Van Vu.
\newblock Stochastic block model and community detection in the sparse graphs:
  A spectral algorithm with optimal rate of recovery.
\newblock \emph{arXiv preprint arXiv:1501.05021}, 2015.

\bibitem[Easley and Kleinberg()]{easleynetworks}
David Easley and Jon Kleinberg.
\newblock Networks, crowds, and markets.
\newblock \emph{Cambridge University}.

\bibitem[Gao et~al.(2013)Gao, Ma, Zhang, and Zhou]{gao2015sbmalgo}
Chao Gao, Zongming Ma, Anderson~Y. Zhang, and Harrison~H. Zhou.
\newblock Achieving optimal misclassification proportion in stochastic block
  model.
\newblock \emph{arXiv preprint arXiv:1312.1733}, 2013.

\bibitem[Girvan and Newman(2002)]{girvan2002community}
Michelle Girvan and Mark~EJ Newman.
\newblock Community structure in social and biological networks.
\newblock \emph{Proceedings of the National Academy of Sciences}, 99\penalty0
  (12):\penalty0 7821--7826, 2002.

\bibitem[Hagen and Kahng(1992)]{hagen1992new}
Lars Hagen and Andrew~B Kahng.
\newblock New spectral methods for ratio cut partitioning and clustering.
\newblock \emph{Computer-aided design of integrated circuits and systems, ieee
  transactions on}, 11\penalty0 (9):\penalty0 1074--1085, 1992.

\bibitem[Hajek et~al.(2015)Hajek, Wu, and Xu]{hajek2015achieving}
Bruce Hajek, Yihong Wu, and Jiaming Xu.
\newblock Achieving exact cluster recovery threshold via semidefinite
  programming: Extensions.
\newblock \emph{arXiv preprint arXiv:1502.07738}, 2015.

\bibitem[Holland et~al.(1983)Holland, Laskey, and
  Leinhardt]{holland1983stochastic}
Paul~W Holland, Kathryn~Blackmond Laskey, and Samuel Leinhardt.
\newblock Stochastic blockmodels: First steps.
\newblock \emph{Social networks}, 5\penalty0 (2):\penalty0 109--137, 1983.

\bibitem[Lei and Rinaldo(2013)]{lei2013consistency}
Jing Lei and Alessandro Rinaldo.
\newblock Consistency of spectral clustering in sparse stochastic block models.
\newblock \emph{arXiv preprint arXiv:1312.2050}, 2013.

\bibitem[Lov{\'a}sz(2012)]{lovasz2012large}
L{\'a}szl{\'o} Lov{\'a}sz.
\newblock \emph{Large networks and graph limits}, volume~60.
\newblock American Mathematical Soc., 2012.

\bibitem[Massouli{\'e}(2014)]{massoulie2014community}
Laurent Massouli{\'e}.
\newblock Community detection thresholds and the weak ramanujan property.
\newblock In \emph{Proceedings of the 46th Annual ACM Symposium on Theory of
  Computing}, pages 694--703. ACM, 2014.

\bibitem[McSherry(2001)]{mcsherry2001spectral}
Frank McSherry.
\newblock Spectral partitioning of random graphs.
\newblock In \emph{Foundations of Computer Science, 2001. Proceedings. 42nd
  IEEE Symposium on}, pages 529--537. IEEE, 2001.

\bibitem[Mossel et~al.(2012)Mossel, Neeman, and Sly]{mossel2012stochastic}
Elchanan Mossel, Joe Neeman, and Allan Sly.
\newblock Stochastic block models and reconstruction.
\newblock \emph{arXiv preprint arXiv:1202.1499}, 2012.

\bibitem[Mossel et~al.(2013)Mossel, Neeman, and Sly]{mossel2013proof}
Elchanan Mossel, Joe Neeman, and Allan Sly.
\newblock A proof of the block model threshold conjecture.
\newblock \emph{arXiv preprint arXiv:1311.4115}, 2013.

\bibitem[Mossel et~al.(2014)Mossel, Neeman, and Sly]{mossel2014consistency}
Elchanan Mossel, Joe Neeman, and Allan Sly.
\newblock Consistency thresholds for binary symmetric block models.
\newblock \emph{arXiv preprint arXiv:1407.1591}, 2014.

\bibitem[Newman(2010)]{newman2010networks}
Mark Newman.
\newblock \emph{Networks: an introduction}.
\newblock Oxford University Press, 2010.

\bibitem[Newman(2003)]{newman2003structure}
Mark~EJ Newman.
\newblock The structure and function of complex networks.
\newblock \emph{SIAM review}, 45\penalty0 (2):\penalty0 167--256, 2003.

\bibitem[Rohe et~al.(2011)Rohe, Chatterjee, Yu, et~al.]{rohe2011spectral}
Karl Rohe, Sourav Chatterjee, Bin Yu, et~al.
\newblock Spectral clustering and the high-dimensional stochastic blockmodel.
\newblock \emph{The Annals of Statistics}, 39\penalty0 (4):\penalty0
  1878--1915, 2011.

\bibitem[Shi and Malik(2000)]{shi2000normalized}
Jianbo Shi and Jitendra Malik.
\newblock Normalized cuts and image segmentation.
\newblock \emph{Pattern Analysis and Machine Intelligence, IEEE Transactions
  on}, 22\penalty0 (8):\penalty0 888--905, 2000.

\bibitem[Van~der Vaart(2000)]{van2000asymptotic}
Aad~W Van~der Vaart.
\newblock \emph{Asymptotic statistics}, volume~3.
\newblock Cambridge university press, 2000.

\bibitem[Van~Mieghem(2006)]{van2006performance}
Piet Van~Mieghem.
\newblock \emph{Performance analysis of communications networks and systems}.
\newblock Cambridge University Press, 2006.

\bibitem[Wasserman(1994)]{wasserman1994social}
Stanley Wasserman.
\newblock \emph{Social network analysis: Methods and applications}, volume~8.
\newblock Cambridge university press, 1994.

\bibitem[Zhang and Zhou(2015)]{supplement}
Anderson~Y. Zhang and Harrison~H. Zhou.
\newblock Supplement to ``minimax rates of community detection in stochastic
  block models''.
\newblock 2015.

\bibitem[Zhao et~al.(2012)Zhao, Levina, Zhu, et~al.]{zhao2012consistency}
Yunpeng Zhao, Elizaveta Levina, Ji~Zhu, et~al.
\newblock Consistency of community detection in networks under degree-corrected
  stochastic block models.
\newblock \emph{The Annals of Statistics}, 40\penalty0 (4):\penalty0
  2266--2292, 2012.

\end{thebibliography}

\newpage
\thispagestyle{empty}
\setcounter{page}{1}
\begin{center}
\MakeUppercase{\large Supplement to ``Mimimax Rates of Community Detection in Stochastic Block Models''}
\medskip

{BY Anderson Y. Zhang and Harrison H.~Zhou}
\medskip

{Yale University}
\end{center}

\appendix

\section{Additional Proofs}
In this appendix we provide the proofs of Theorem \ref{thm:lower_beta} and Theorem \ref{thm:upper_beta}.

\subsection{Proof of Theorem \ref{thm:lower_beta}}
~\\
\indent\textit{(1)} For $K=2$, the least favorable case for $\Theta$ is still $\Theta^0$. The proof is identical to that of Theorem \ref{thm:lower}.

\textit{(2)} For $K=3$, it is always possible to have $\sigma\in\Theta$ such that a constant proportion of communities have size $\lfloor \frac{n}{\beta K}\rfloor$, and another constant proportion have the same size $\lceil \frac{n}{\beta K}\rceil$, with the rest communities have much larger size. Define $\Theta^L$ to contain all such $\sigma$. Then with identical arguments used to establish Lemma \ref{lem::multiple_test} and Lemma \ref{lem::binomial_tail} we have
\begin{align*}
\inf_{\hat\sigma}\sup_{\Theta}\E r(\sigma,\hat\sigma)&\geq\inf_{\hat\sigma}\sup_{\sigma\in\Theta^L}B_\tau(\hat\sigma(1))\\&
\geq c\p\Big(\sum_{u=1}^{\lfloor n/\beta K\rfloor}X_u\geq \sum_{u=1}^{\lfloor n/\beta K\rfloor}Y_u\Big)\\&
\geq \exp(-(1+o(1))nI/\beta K).
\end{align*}

\subsection{Proof of Theorem \ref{thm:upper_beta} ($K=2$)}

Without loss of generality we assume $\frac{n}{2}=\lfloor \frac{n}{2}\rfloor$ throughout this section. For arbitrary $\sigma,\sigma_0\in\Theta$ with $d(\sigma,\sigma_0)=m$, we can define $\alpha(\sigma;\sigma_0)$ and $\gamma(\sigma;\sigma_0)$ the same way as in Section \ref{subsec:proof_upper}. Note that $m\leq \frac{n}{2}$ since $d(\sigma,\sigma_0)=\min\{d_H(\sigma,\sigma_0),n-d_H(\sigma,\sigma_0)\}$.  By Proposition \ref{prop::chernoff}, we have
\begin{align*}
\p(T(\sigma)\geq T(\sigma))&\leq\p\left(\sum_{i=1}^{\gamma}X_{i}-\sum_{i=1}^{\alpha}Y_{i}\geq \lambda(\gamma-\alpha)\bigg|X_{i}\stackrel{iid}{\sim} \text{Ber}(\frac{b}{n}),Y_{i}\stackrel{iid}{\sim} \text{Ber}( \frac{a}{n} )\right).
\end{align*}
Note that in $K=2$ we have a specific equality as $\alpha+\gamma=m(n-m)$. Recall that $\lambda=-\frac{1}{2t^\star}\log\big(\frac{\frac{a}{n}\exp(-t^\star)+1-\frac{a}{n}}{\frac{b}{n}\exp(t^\star)+1-\frac{b}{n}}\big)$. By the Chernoff bound,
\begin{align*}
\p(T(\sigma)\geq T(\sigma_0))&\leq \big(\E e^{t^\star X_i}\big)^\gamma\big(\E e^{-t^\star Y_i}\big)^\alpha e^{-t^\star\lambda(\gamma-\alpha)}\\
&=\big(\E e^{t^\star X_i}\E e^{-t^\star Y_i}\big)^{\frac{m(n-m)}{2}}\Big(\frac{\E e^{t^\star X_i}}{\E e^{-t^\star Y_i}}e^{-2t^\star\lambda'}\Big)^{\gamma-\frac{m(n-m)}{2}}\\
&=\exp\Big(-\frac{m(n-m)I}{2}\Big),
\end{align*}
where we use $\E e^{t^\star X_i}\E e^{-t^\star Y_i}=\exp(-I)$ and $e^{2t^\star\lambda'}=\frac{\E e^{t^\star X_i}}{\E e^{-t^\star Y_i}}$. The proof is similar to that of Theorem \ref{thm:upper}. Here we only include the key technique and omit the details. Assume $0<\epsilon<1/8$. Consider the following three cases:
\newline
\textit{(1)} If $nI/2>(1+\epsilon)\log n$, define $m_0=1$ and $m'=\epsilon n/2$. Then $P_1\leq n\exp(-(n-1)I/2)$. Denote $R=n\exp(-(n-1)I/2)$. We have
\begin{align*}
P_m\leq
\begin{cases}
(\frac{2en}{2})^m\exp(-\frac{m(n-m)I}{2})\leq Rn^{-\epsilon m/4},\text{ for }m_0<m\leq m'\\
(\frac{2en}{\epsilon n})^m\exp(-\frac{nmI}{4})\leq R\exp(-\frac{n(m-4)I}{8}),\text{ for }m'<m\leq n/2.
\end{cases}
\end{align*}
Then $n\R(\sigma,\hat\sigma)\leq \sum_{m=1}^{n/2}mP_m=(1+o(1))R$.
\newline
\textit{(2)} If $nI/2<(1-\epsilon)\log n$, define $m_0=n\exp(-(1-e^{-\epsilon nI/2})nI/2)$ and $m'=n\exp(-nI/8)$. We have
\begin{align*}
P_m\leq
\begin{cases}
(\frac{2en}{m_0})^m\exp(-\frac{m(n-m')I}{2})=\exp(-e^{-\frac{\epsilon nI}{2}}\frac{nmI}{4}),\text{ for }m_0<m\leq m',\\
(\frac{2en}{m'})^m\exp(-\frac{nmI}{4})\leq \exp(-\frac{nmI}{16}),\text{ for }m'<m\leq n/2.
\end{cases}
\end{align*}
Then $\R(\sigma,\hat\sigma)\leq m_0/n + \sum_{m>m_0}^{n/2}P_m=(1+o(1))m_0/n$.
\newline
\textit{(3)} If $\frac{nI}{2\log n}\rightarrow 1$, there exists a positive sequence $w\rightarrow 0$ such that $|\frac{nI}{2\log n}-1|\ll w$ and $\frac{1}{\sqrt{\log n}}\leq w$. Define $m_0=n\exp(-(1-w)nI/2)$ and $m'=w^2n$. 
\begin{align*}
P_m\leq
\begin{cases}
(\frac{2en}{m_0})^m\exp(-\frac{m(n-m')I}{2})\leq \exp(-\frac{wnmI}{4}),\text{ for }m_0<m\leq m'\\
(\frac{2en}{m'})^m\exp(-\frac{nmI}{4})\leq\exp(-\frac{nmI}{8}),\text{ for }m'<m\leq n/2.
\end{cases}
\end{align*}
Then $\R(\sigma,\hat\sigma)\leq m_0/n + \sum_{m>m_0}^{n/2}P_m=(1+o(1))m_0/n$.

\subsection{Proof of Theorem \ref{thm:upper_beta} ($K\geq 3$)}
For the upper bound, we need the following lemma in replace of Lemma \ref{lem::number_black_white}. Other than that, the proof is identical to that for Theorem \ref{thm:upper} and thus omitted.
\begin{lemma}\label{lem::number_black_white_extend}
Assume $1\leq \beta<\sqrt{\frac{5}{3}}$. Let $\sigma\in\Theta$ be an arbitrary assignment satisfying $d(\sigma,\sigma_{0})=m$, where $0<m<n$ is a positive integer. Then
\begin{align*}
\alpha(\sigma;\sigma_0)\wedge \gamma(\sigma;\sigma_0)\geq
\begin{cases}
\frac{nm}{K\beta}-m^2, \text{ if }m\leq \frac{n}{2K},\\
\frac{c_\beta nm}{K}, \text{ if }m> \frac{n}{2K},
\end{cases}
\end{align*}
where $c_\beta=\frac{(5-3\beta^2)^2}{2\beta(1+3(5-3\beta^2)^2)}$.
\end{lemma}
\begin{proof}[Proof of Lemma \ref{lem::number_black_white_extend}]
It is sufficient to show the equality for $\gamma(\sigma;\sigma_0)$. First consider the case $m\leq\frac{n}{2\beta  K}$. Without loss of generosity, let $\sigma$ satisfy
\begin{align*}
\sigma(i)=k,\forall i\in\Bigg[\sum_{j=1}^{k-1}n'_{j}+1,\sum_{j=1}^{k}n'_{j}\Bigg].
\end{align*}
Here $\{n'_{k}\}$ are sizes of all communities in $\sigma$. Assume $d_H(\sigma,\sigma_0)=m$, then $m=|\{i:\sigma(i)\neq\sigma_0(i)\}|$. Define $m_k=|\{i:\sigma(i)=k,\sigma_0(i)\neq k\}|$ then $m=\sum_k m_k$. For $k\in[K]$, define
\begin{align*}
\gamma_{k}(\sigma;\sigma_0)&=|\{(i,j):\sigma(i)=\sigma(j)=k,\sigma_0(i)\neq \sigma_0(j),i<j\}| \\
&=\Big|\Big\{(i,j):\sigma_0(i)\neq \sigma_0(j),\sum_{j=1}^{k-1}n'_{j}+1\leq i<j\leq \sum_{j=1}^{k}n'_{j}\Big\}\Big|.
\end{align*}
We see that $\gamma(\sigma;\sigma_0)=\sum_{k=1}^{K}\gamma_{k}(\sigma;\sigma_0)$. We have $m_k\leq \frac{n}{2\beta K}\leq \frac{n'_k}{2}$, and also $\gamma_{k}(\sigma;\sigma_0)\geq |\{i:\sigma(i)=k, \sigma_0(i)=k\}||\{i:\sigma(i)=k, \sigma_0(i)\neq k\}|= m_k(n_k-m_k)$. Then
\begin{align*}
\gamma(\sigma;\sigma_0)\geq \sum_k m_k(n_k-m_k)\geq \frac{mn}{\beta K}-m^2.
\end{align*}
Now consider the case that $m>\frac{n}{2\beta K}$. Define $m_{k,k'}=|\{i:\sigma(i)=k,\sigma_0(i)= k'\}|$ for any $k,k'\in[K]$. We see that equations $m_k=\sum_{k'\neq k}m_{k,k'}$ and $n'_k=m_k+m_{k,k}$ and $n_{k'}=\sum_{k}m_{k,k'}$ hold for all $k,k'\in[K]$.

For each $k\in[K]$, we want to get the value of $\gamma_{k}(\sigma;\sigma_0)$. We divide $k\in[K]$ into the following three categories:

(1) We say $k\in\mathcal{K}_1$ if for all $k'\neq k$, $m_{k,k'}\leq \frac{2}{3}n'_k$. For a given $m_k$, we have
\begin{align*}
\frac{\gamma_{k}(\sigma;\sigma_0)}{n'_k m_k}=\frac{\frac{1}{2}(n_k^{'2}-\sum_{k'}m_{k,k'}^2)}{n'_k m_k},
\end{align*}
with $m_k=\sum_{k'\neq k}m_{k,k'}$. When $m_k\leq \frac{2}{3}n'_k$, it is easy to check
\begin{align*}
\frac{\gamma_{k}(\sigma;\sigma_0)}{n'_k m_k}\geq \frac{\frac{1}{2}(n_k^{'2}-(n'_k-m_k)^2-m_k^2)}{n'_k m_k}=\frac{n'_k-m_k}{n'_k}\geq \frac{1}{3}.
\end{align*}
When $m_k> \frac{2}{3}n'_k$,
\begin{align*}
\frac{\gamma_{k}(\sigma;\sigma_0)}{n'_k m_k}&\geq \frac{\frac{1}{2}(n_k^{'2}-(n'_k-m_k)^2-(m_k-\frac{2}{3}n'_k)^2-(\frac{2}{3}n'_k)^2)}{n'_k m_k}\\
&\geq \frac{m_k(n'_k-m_k)+\frac{2}{3}n'_k(m_k-\frac{2}{3}n'_k)}{n'_k m_k}\\
&\geq \frac{2}{9}.
\end{align*}
Thus $\gamma_{k}(\sigma;\sigma_0)\geq \frac{2nm_k}{9\beta K}$ in both cases.

(2) We say $k\in\mathcal{K}_2$ if exists $k'\neq k$ such that $m_{k,k'}>\frac{2}{3}n'_k$. Claim $m_{k',k'}>\frac{1}{3}n'_k$. Otherwise, from $\sigma$ we can exchange the labels $k$ and $k'$ to get a new estimator $\sigma'$. This helps to correctly recover at least $m_{k,k'}-m_{k,k}-m_{k',k'}>\frac{2}{3}n'_k-\frac{1}{3}n'_k-\frac{1}{3}n'_k>0$ more nodes. Since $\sigma'\in\Gamma(\sigma)$, this implies $m=d(\sigma_0,\sigma)\leq d_H(\sigma_0,\sigma')<d_H(\sigma_0,\sigma)=m$, which leads to a contradiction.

On the other hand, we have $m_{k'}=n'_{k'}-m_{k',k'}\geq n'_{k'}-(n_{k'}-m_{k,k'})\geq \frac{n}{\beta K}-\frac{\beta n}{K}+\frac{2n}{3\beta K}\geq \frac{(5-3\beta^2)n}{3\beta K}> 0$. This implies
\begin{align*}
\frac{\gamma_{k}(\sigma;\sigma_0)+\gamma_{k'}(\sigma;\sigma_0)}{m_k+m_{k'}}\geq \frac{\gamma_{k'}(\sigma;\sigma_0)}{m_k+m_{k'}}\geq \frac{m_{k',k'}m_{k'}}{m_k+m_{k'}}\geq \frac{\frac{1}{3}n'_k}{\frac{m_k}{m_{k'}}+1}\geq \frac{\frac{n}{3\beta K}}{\frac{\beta}{(5-3\beta)^2/(3\beta)}+1}.
\end{align*}
Thus we have $\gamma_{k}(\sigma;\sigma_0)+\gamma_{k'}(\sigma;\sigma_0)\geq \frac{2c_\beta n(m_k+m_{k'})}{K}\geq \frac{2c_\beta nm_k}{K}$.

Apparently $[K]=\mathcal{K}_1\cup \mathcal{K}_2$ and $\mathcal{K}_1\cap \mathcal{K}_2=\emptyset$. Claim for any $k\in\mathcal{K}_1$, there exists at most one $k'\neq k$ such that $m_{k',k}> \frac{2}{3}n'_{k'}$. Otherwise if there exists another $k''\neq k'$ such that $k''\neq k$ and $m_{k'',k}> \frac{2}{3}n'_{k''}$. Since $k',k''\in\mathcal{K}_2$, this leads to $m_{k,k}\geq \frac{1}{3}(n'_k\vee n'_{k'})$. Then $n_{k}\geq m_{k',k}+m_{k'',k}+m_{k,k}>n'_{k'}+\frac{2}{3}n'_{k''}\geq \frac{5n}{3\beta K}>\frac{\beta n}{K}$ which leads to a contradiction. Note that $c_\beta\leq \frac{2}{9}$. Thus
\begin{align*}
\gamma(\sigma;\sigma_0)&=\frac{1}{2}\sum_{k\in[K]}2\gamma_k(\sigma;\sigma_0)\\
&\geq\frac{1}{2}\left(\sum_{k\in\mathcal{K}_1}\frac{2nm_k}{9\beta K}+\sum_{k\in\mathcal{K}_2}\frac{2c_\beta nm_k}{K}\right)\\
&\geq \frac{c_\beta nm}{K}.
\end{align*}
\end{proof}

\end{document}